\documentclass[12pt]{amsart}

\usepackage{amsmath,amsthm,amsopn,amsfonts,amssymb,color}
\usepackage[colorlinks=true,linkcolor=blue,urlcolor=blue]{hyperref} 
\usepackage{cancel}
 \usepackage{tikz}
\usetikzlibrary{intersections,calc}
\usepackage{graphicx}
\usepackage{float}
\usepackage{amsmath}
\usepackage{fancyhdr}
\usepackage{epstopdf}
\usepackage{amsfonts}
\usepackage{pgfplots}
\pgfplotsset{compat=newest}
\usetikzlibrary{intersections, pgfplots.fillbetween}
\pgfdeclarelayer{pre main}
\usepackage{amsmath}

\usepackage{booktabs}
\usepackage{cancel}
\usepackage{times}
\usepackage{mathrsfs}
\usepackage{multirow,multicol}
\usepackage{amsfonts}
\usepackage{amsthm}
\usepackage{amsmath}
\usepackage{amsmath,amssymb}
\newtheorem{teo}{Theorem}[section]
\newtheorem{prop}[teo]{Proposition}
\newtheorem{defin}[teo]{Definition}
\newtheorem{remark}[teo]{Remark}

\newtheorem{lemma}[teo]{Lemma}

\newcommand{\dual}{L^{2}(0,T,W^{-1,2}(\Omega))}

\newcommand{\pspacegen}{L^p\big(0,T,X\big)}
\newcommand{\pspacegeninfty}{L^\infty\left(0,T,X\right)}
\newcommand{\parspgencont}{C^0\left([0,T],X\right)}
\newcommand{\LBoch}[5]{L^{#1}\left(#2,#3 ; L^{#4} (#5)\right)}

\newcommand{\Boch}[6]{L^{#1}\left(#2,#3 ; W_0^{#4,#5} (#6)\right)}

\newcommand{\dx}{\,\mathrm d }

\newcommand{\Om}{\Omega}

\DeclareMathOperator*{\esssup}{ess\sup}

\def\elle#1{L^{#1}(\Omega)}

\def\div{{\rm div}}
\def\elle#1{L^{#1}(\Omega)}

\def\io{\int_{\Omega}}

\def\iqT{\int_{\Omega_{T}}}
\def\norma#1#2{\|#1\|_{\lower 4pt \hbox{$\scriptstyle #2$}}}
\def\un{u_n}

\def\finedim
\def\gw{G_{\tilde{k}}(w_n)}

\def\elle#1{L^{#1}(\Omega)}

\def\eps{\varepsilon}

\def\be{\begin{equation}}
\def\ee{\end{equation}}

\usepackage{todonotes}

 \usepackage{color}
 \numberwithin{equation}{section}

\title[Well-posedness results for superlinear Fokker-Planck equations]{Well-posedness results for superlinear Fokker-Planck equations}

	\author[Stefano Buccheri]{Stefano Buccheri}
\thanks{S.B. -  Dipartimento di Matematica e Applicazioni ``R.
Caccioppoli",  Universit\`a  degli Studi di Napoli Federico II,
Via Cintia, 80126, Napoli,
Italy, stefano.buccheri@unina.it}

\author{ Fernando Farroni}
\thanks{F.F. -  Dipartimento di Matematica e Applicazioni ``R.
Caccioppoli",  Universit\`a  degli Studi di Napoli Federico II,
Via Cintia, 80126, Napoli,
Italy, fernando.farroni@unina.it}

\author{ Gabriella Zecca}
\thanks{G.Z. - Dipartimento di Matematica e Applicazioni ``R.
Caccioppoli",  Universit\`a  degli Studi di Napoli Federico II,
Via Cintia, 80126, Napoli,
Italy, g.zecca@unina.it}

\begin{document}

\maketitle

\begin{abstract}
In this manuscript we deal with a class of nonlinear Fokker-Planck equations with the following structure
\[
\partial_t u - \div\big(M\nabla u+ E h(u)\big)=0,
\]
with $M$ a bounded elliptic matrix,  $E$ a vector field in a suitable Lebesgue space, and $h(u)$ featuring a superlinear growth for $u$ large. We provide existence results of $C([0,T),L^1)$ distributional solutions to initial-boundary value problems related to the equation above together with some qualitative properties of solutions.
\end{abstract}

\medskip

\tableofcontents

 
  
\section{Introduction}
In this manuscript we deal with a class of nonlinear Fokker-Planck equations with the following structure
\begin{equation}\label{KQ}
\partial_t u - \div\big(M\nabla u+ E h(u)\big)=0,
\end{equation}
with $M$ a bounded elliptic matrix,  $E$ a vector field in a suitable Lebesgue space, and $h(u)$ featuring a superlinear growth for $u$ large. 

Taking $M$ the identity matrix, $E(x)=x$, and $h(u)=u(1+u)$, we recover the equation proposed in \cite{KQ1,KQ2} to describe the dynamic of Bose-Einstein particles, where the quadratic drift term accounts for an extra concentration phenomenon due to quantum effects. Under this specific set of assumptions, equation \eqref{KQ} has been addressed using approaches based on moments estimates or exploiting its variational structure; see for instance \cite{CCLR,CdFT,Hop,toscani} and references therein. Among these results, let us stress that Toscani showed in \cite{toscani} that, if the $L^1$ norm of the initial datum is large enough, finite time blow up occurs and measured valued solutions appear. Therefore, existence of global in time smooth solutions is not expected in general for \eqref{KQ}. 
However, the previously mentioned results hold either for specific choices for the vector field $E$ and the matrix $M$ or under strong structural assumptions on it (for instance $E=\nabla \phi$) and does not seem that they can be immediately adapted to more general cases. Let us also quote \cite{oned2,oned1}, where a careful analysis of a initial-boundary value problem associated to \eqref{KQ} is carried out in $\mathbb{R}$.\\

Our contribution here is to set up a well posedness theory for $C([0,T),L^1)$ solutions to initial-boundary values problem related to \eqref{KQ} under fairly general assumptions on the coefficients. The approach we follow is entirely not variational, it does not need specific assumptions on the structure of the vector field $E$, and it does not make use of any representation formula. Let us stress that here we will not deal with measured valued solutions (for more details on the topic see for instance \cite{Hop}). We will rather keep our analysis either in cases in which global existence holds, or proposing local in time results, namely before that the possible \emph{condensation} of the solution to a measure may occur (see again \cite{toscani}).\\

For us $\Omega$ is an open bounded subset of $\mathbb{R}^N$, with $N\ge3$, and  $\Omega_T$ denotes the space-time cylinder $(0,T)\times\Om$ for any given $T>0$. The measurable matrix $M(t,x)$ satisfies, for almost every $(t,x)\in \Omega_T$ and for every $\xi\in\mathbb R^N $, the following conditions
\begin{equation}\label{mcon}
\alpha |\xi|^2\leq 
  M(t,x)\,\xi \xi, \quad |M(t,x)| \leq\beta,\\  
\end{equation}
for two positive constants $0<\alpha\le\beta$, while the vector field $E:\Omega\to \mathbb{R}^N$ and the function $u_0:\Omega\to \mathbb{R}$ belong to suitable Lebesgue spaces. We focus on the following problem
\begin{equation}\label{introintro}
\begin{cases}
\partial_t u-\div\left(M(t,x)\nabla u +E(t,x) |u|^{\theta} u \right)=0 & \quad  \mbox{in } \Omega _T,\\
u = 0 & \quad \mbox{on }  (0,T) \times \partial \Omega ,\\
u(x,0)=u_0(x) & \quad \mbox{in } \Omega.
\end{cases}
\end{equation}
 The main assumption here is that $\theta >0$, namely the convection term is superlinear. Our approach also works for nonlinearities that satisfy $|h(u)|\le C(1+|u|^{\theta+1})$, we have chosen $h(u)=|u|^{\theta}u$ for clarity of exposition. The notion of weak solution that we consider is given in Definition \ref{defsol1} below.\\

For $\theta =0$ existence and regularity properties of the differential operator 
\[
u\to -\div\big(M(t,x)\nabla u +E(t,x)  u \big)
\]
have been studied both in the stationary \cite{BOMbis,lucio, buccheri,cirmibis,  delvecchio,FGMZcv, Mos} and in the evolutionary framework \cite{bop,BOP,nodea,anona,FAMO18,por}. For superlinear drift  we refer to \cite{BBC} (see also \cite{fessel,daSilva}). The leitmotiv of these works is that the (super)linear drift makes the differential operator not coercive and therefore the usual strategies to get a priori estimate fail. Depending on the assumptions, the effect of the lower order term may be balanced by the second order elliptic operator (see for instance \cite{lucio}) or may represent an obstruction to existence/regularity of solutions (see for instance  \cite{BBC}).

Let also mention that if the drift term is on the form $-$div$(\Phi(u))$ it is possible to obtain existence of \emph{renormalized} solution for any continuous $\Phi:\Omega\to \mathbb{R}^N$ with no growth conditions at infinity, see \cite{BGDM,porr1999}. However these results do no apply in our contest of space dependent drifts. \\

At first we provide global in time existence and uniqueness of solutions for $\theta$ below a given threshold, namely,
\begin{equation}\label{thetasmall}
0<\theta\le \frac{1}{N}.
\end{equation}
\begin{teo}\label{thetasmalf=0}
Let us assume \eqref{mcon}, \eqref{thetasmall}, and take $r\in(N,\infty]$ such that
\be \label{thetasmallbis}
\frac1r+\theta=\frac{1}{N}.
\ee
Then, for any given $T>0$, $u_0\in L^\mu(\Omega)$ with $\mu\ge1$, and $|E|\in L^{\infty}((0,\infty);L^r(\Omega))$, there exists a unique weak solution  $u\in C([0,T];L^1(\Omega))\cap L^2_{loc}((0,T); W^{1,2}_0(\Omega))$ to \eqref{introintro} such that 
\be\label{intro09:36bis}
\|u(t)\|_{\elle{\mu}}\le C_1 \|u_0\|_{\elle{\mu}}  \qquad \mbox{ for }t\in (0,T),
\ee
with $C_1=C_1(\alpha, N, \theta, \mu, r, T, E, \|u_0\|_{\elle 1})$.
Moreover for any $m\ge\mu$ there exists $ C_2= C_2(\alpha, N, \theta, \mu, m, r,T,E,\|u_0\|_{\elle1})$  such that 
\be\label{contractiveintro1}
\|u(t)\|_{L^{m} (\Omega )}\leq \|u_0\|_{L^{\mu}(\Omega)} \frac{ C_2}{t^{\frac N2(\frac{1}{\mu}-\frac 1m) } },\qquad \mbox{ for }t\in (0,T).
\ee
\end{teo}
On the other hand, if we drop assumption \eqref{thetasmall}, the result reads as follows.
\begin{teo}\label{thetabigf=0}
Given $\theta>0$, take $\mu\in(1,\infty)$ and $r\in(N,+\infty]$ such that
\be\label{thetabig}
\frac1r+ \frac{\theta}{\mu}<\frac{1}{N}.
\ee
For any $u_0\in \elle{\mu}$ and $|E|\in \big(L^{\infty}((0,\infty);\elle{r})\big)^N$ there exists $T^{*}>0$ depending on $\alpha, N,\mu, r, \theta, \|E\|_{L^{\infty}(0,T,L^{ r}(\Omega))}$, and $\|u_0\|_{\elle{\mu}}$, such that, for $T= T^*$, problem \eqref{introintro} admits a unique weak solution $u$ in  $C([0,T^*);L^1(\Omega))\cap L^2_{loc}((0,T^*); W^{1,2}_0(\Omega))$ that satisfies
\be\label{intro09:36}
\|u(t)\|_{\elle{\mu}}\le \|u_0\|_{\elle{\mu}}\left(\frac{T^*}{T^*-t }\right)^{\gamma} \qquad \mbox{ for } t\in (0,T^*).
\ee
Moreover, the following contractive type estimate holds true: for any $m>\mu$ there exists $C(t)=C(\alpha, N, \theta, m, \mu, r, \|E\|_{L^{\infty}(0,T^*,L^{ r}(\Omega))},\|u_0\|_{L^{\mu}(\Omega)},t)$ such that
\be\label{17:38}
\|u(t)\|_{L^m (\Omega )}\leq  \frac{C(t)}{t^{\frac N2 (\frac{1}{\mu}-\frac 1m)}} \qquad \mbox{ for } t\in (0,T^*).
\ee
with $C(t)\to+\infty$ as $t\to T^*$
\end{teo}

While Theorem \ref{thetasmalf=0} is valid for any given $T>0$, Theorem \ref{thetabigf=0} gives existence only on the finite time interval $(0,T^*)$, with estimate \eqref{intro09:36} blowing up as $t\to T^*$.
Even if the two proofs follow slightly different procedures, the main idea behind these theorems is the same: to obtain a differential inequality satisfied by a suitable norm of the solution $u$ (see for instance \cite{BOP,BZ,MP,Porr2001}). A formal computation shows that, if $u$ is the solution to \eqref{introintro} with initial datum in $\elle{\mu}$, the quantity $y(t)=\|u(t)\|_{\elle \mu}^{\mu}$ solves
\be\label{eqintro}
y'\le Q y^{1+d},
\ee
for suitable constant $Q$ and exponent $d$. This differential inequality is obtained multiplying the equation in \eqref{introintro} by a suitable test function, integrating in space, and coupling Sobolev embedding (in space) with H\"older and interpolation inequalities. Assumptions \eqref{thetasmallbis} and \eqref{thetabig} naturally arise by the dimensional analysis associated to this process. If \eqref{thetasmallbis} is valid, it follows that $d=0$ (despite the fact that $\theta>0$) and we have an estimate for all times. On the other hand, if $\theta$ is large, we have $d>0$, and the dynamic associated to \eqref{eqintro} gives rise to finite time blow-up of the estimate. All the details to make this argument rigorous and the proof of Theorems \ref{thetasmalf=0} and \ref{thetabigf=0} can be found in Section \ref{differentialineq}.

For the sake of clarity, let us stress that the information given by \eqref{intro09:36} is only a bound from above for $\|u(t)\|_{\elle{\mu}}$, that does not imply blow-up for the solution. However, we recall the possibility of measured valued solution appearing in finite time is in accordance with \cite{toscani}.

It is moreover interesting to note that, while the nonlinear drift may exhibit a singular behavior for sufficiently large times, for small values of $t$ the regularizing effect of second order operator (the diffusion) dominates the drift for any $\theta$. Indeed the exponent $\frac{N}{2}\left({1}/{\mu}-1/m\right)$ in \eqref{17:38} is exactly the same as in the classical heat equation (see \cite{porzio1} and reference therein).\\

We also provide existence and $L^p$-regularity results for a nonzero right hand side. Precisely, we consider the following problem
\begin{equation} \label{problemintro}
\begin{cases}
\partial_t u-\div\left(M(t,x)\nabla u +E(t,x) |u|^{\theta} u \right)=f(t,x) & \quad  \mbox{in } \Omega _T,\\
u = 0 & \quad \mbox{on }  (0,T) \times \partial \Omega,\\
u(x,0)=u_0(x) & \quad \mbox{in } \Omega.
\end{cases}
\end{equation}
In order to state our results, we introduce the following quantities:
\begin{equation}\label{qstardef}
q^{\star}=\frac{(N+2)q}{N+2-q} \ \ \ \mbox{and} \ \ \ q^{\star\star}=\frac{(N+2)q}{N+2-2q}.
\end{equation}
Observe that $q^{\star}$ and $q^{\star\star}$ are the natural parabolic counterparts of the Sobolev coniugate exponents of $q$ in the elliptic settings.
\begin{teo}\label{exist2}
Assume $\sigma=2\frac{N+2} N$,
$\theta\geq 0$, $f\in L^q(\Omega_T)$, $u_0\in L^{q^{\star\star}\frac{N}{N+2}}(\Omega)$, $E\in L^r(\Omega_T)$ with
\begin{equation}\label{closelin}
\frac1r+\theta=\frac1{N+2}.
\end{equation}
$(i)$ If $q=1$, problem \eqref{problemintro} admits a weak solution $u\in L^{m}(0,T,  W^{1,m}_0(\Omega))\cap L^{m^{\star}}(\Omega_T)$ for all $m\in[1,1^{\star})$.\\ 
$(ii)$ If $q\in (1, \sigma^\prime)$, problem \eqref{problemintro} admits a weak solution $u\in L^{q^{\star}}(0,T,  W^{1,q^{\star}}_0(\Omega))\cap L^{q^{\star\star}}(\Omega_T)$.\\ 
$(iii)$ If $q\in [\sigma^\prime,\frac{N+2}{2})$, problem \eqref{problemintro} admits a unique weak solution $u$ in $L^{2}(0,T, W^{1,2}_0(\Omega))\cap L^{q^{\star\star}}(\Omega_T)$.
\end{teo}
Again, we see that there exists a range of values of $\theta$ for which the influence of the drift, that is still superlinear, is completely absorbed by the principal operator. Indeed, we have existence for all times and we recover the same type of regularity as in the case $E\equiv0$.\\
Here we use the test-function based approach developed in \cite{BG,BDGO2}.
However, let us mention that even for such small values of $\theta$, the achievement of the a priori estimates is not standard and we had to adopt a slicing procedure in time to deal with the non coercive drift term.\\

As before, we also consider cases where the exponent $\theta$ is large.
\begin{teo}\label{thmSte} For any $\theta>0$, take $q\in [2\frac{N+2}{N+4},\frac{N+2}{2})$ and $r>N+2$ such that
\begin{equation}\label{farfromlin}
\frac1r+\frac{\theta}{q^{\star\star}}=\frac{1}{N+2},
\end{equation}
with $q^{\star\star}$ defined in \eqref{qstardef}, and assume  that
\[
f\in L^q(\Omega_T),  \ \ \ E\in (L^r(\Omega_T))^N, \ \ \ u_0\in \elle{q^{\star\star}\frac{N}{N+2}}.
\]
Then there exists a constant $C=C(\alpha, q)$ such that, whenever
\be\label{smallness}
\|E\|_{L^{r}(\Om_T)}^{\frac1{\theta}}\left[\|f\|_{L^{q}(\Om_T)} + \|u_0\|_{L^{q^{\star\star}\frac{N}{N+2}}(\Omega)}\right]\le \frac{\theta}{(C(\theta+1))^{\frac{\theta+1}{\theta}}},
\ee
problem \eqref{problemintro} admits a unique weak solution $u$ in $L^{q^{\star\star}}(\Omega_T)\cap L^2(0,T,W^{1,2}_0(\Omega)) $.
\end{teo}

To prove this result, we apply Schauder fixed point Theorem to the map $\mathscr F: L^{q^{\star\star}}(\Omega_T)\to  L^{q^{\star\star}}(\Omega_T)$, that associates to each $v\in  L^{q^{\star\star}}(\Omega_T)$ the unique solution $u\in L^{q^{\star\star}}(\Omega_T)$ to
\[
\partial_t u-\div\left(M(t,x)\nabla u\right)= \div\left(E(t,x) |v|^{\theta} u \right)+f(t,x)  \quad  \mbox{in } \Omega _T, 
\]
with $u=0$ on $(0,T)\times\partial\Omega$ and $u(x,0)=u_0(x)$ in $\Omega$ (see \cite{BDGO2} for the existence of such a solution). So, the smallness assumption \eqref{smallness} on the size of the data assures existence of an invariant ball for $\mathscr F$. Notice that this existence for small data is somehow connected to the local in time existence of Theorem \ref{thetabigf=0}. Indeed, taking $f\equiv0$ for simplicity, \eqref{smallness} is always verified for $T$ sufficiently small.\\
As before, assumptions \eqref{closelin} and \eqref{farfromlin} appear naturally from a dimensional analysis of the equation and allow to close the a priori estimates in Theorem \ref{exist2} and to have the map $\mathscr F$ well defined in Theorem \ref{thmSte}. Roughly speaking, the difference with \eqref{thetasmallbis} and \eqref{thetabig} is due to the use of the Gagliardo-Nirenberg inequality on the parabolic cylinder instead of Sobolev inequality in space.\\

We conclude this introduction by pointing out that in Section \ref{preliminari} (see Theorem \ref{comparison} below) we prove an $L^1$ contraction estimate (see the for instance \cite{otto} and reference therein) of the following type
\[
\int_{\Omega}(v(\tau)-w(\tau))_+\le\int_0^{\tau}\int_{\Omega}(f-g)\chi_{v>w}+\int_{\Omega}(v_0-w_0)_+, \quad \tau \in[0,T],
\]
where $v$ is the solution associated to source $f$ and initial datum $v_0$ and $w$ is the solution associated to source $g$ and initial datum $w_0$. This kind of comparison principle represents a key tool to derive the uniqueness and to obtain the well posedness in  $C([0,T), L^1)$ of our initial boundary value problem.

\section{Notation and functional framework}

In order to introduce our functional framework, take $T>0$, $1\leq p \leq \infty$, and let $X$ be a separable Banach space endowed with a norm $\|\cdot\|_X$. The 
space 
$
\pspacegen
$  
is given by all measurable
functions $u\colon [0,T] \rightarrow X$ that are $L^p$ integrable. Such a space is a  Banach space equipped with the norm
\[
\|u\|_{\pspacegen}
:=\left(
\int_0^T \|u(t)\|^p_X \dx t
\right)^{1/p}
\]
for $1\leq p <\infty$ and
\[
\|u\|_{\pspacegeninfty}
:=\esssup_{0<t<T }  
 \|u(t)\| _X
\]
for $p=\infty$. As usual we also define
\[
L_{\mathrm{loc}}^p(0,T, X)
:=
\left\{
u : [0,T] \to X \; \text{measurable} \; : \;
u \in L^p(t_1,t_2; X) \text{ for every } [t_1,t_2] \subset (0,T)
\right\}.
\]
In other words, a function $u$ belongs to $L_{\mathrm{loc}}^p(0,T, X)$ if it is $L^p$ integrable
with values in $X$ on every compact subinterval of $(0,T)$. Analogously,  the space 
$
\parspgencont
$ consists in the set of continuous functions $u\colon [0,T] \rightarrow X$, 
equipped the norm  
\[
\|u\|_{\parspgencont}
:=\max_{0\leq t \leq T}
 \|u(t)\|_X.
\]
We consider now the  Banach space  
\[
W^2(0,T):=
\left\{
u \in L^2(0,T,W_0^{1,2}(\Om))\colon\, u_t \in {\dual}
\right\},
\]
endowed with the norm
\[
\|
u
\|_{W^2(0,T)}
 := 
\|
u
\|_{L^2(0,T,W^{1,2}(\Om))}
+
\|
u_t
\|_{\dual},
\]
and its local counterpart
\[
W_{\mathrm{loc}}^2(0,T)
:=
\left\{
u \in L_{\mathrm{loc}}^2(0,T, W_0^{1,2}(\Omega))
:\;
u_t \in L_{\mathrm{loc}}^{2}(0,T, W^{-1,2}(\Omega))
\right\}.
\]

%
%
%
%
%
Given $1<p<\infty$, the Marcinkiewicz space $\mathcal M^p(\Omega_T)$ consists of all measurable functions $f$ defined on $\Omega_T$ such that
\begin{equation*}\label{2.1}
\|f\|^p_{\mathcal M^p}:=\sup_{k>0}k^{ p} \lambda_f(k)
\end{equation*}
is finite, with $\lambda_f(k):=  \left| \left\{ (t,x)\in \Om_T: |f(t,x)|>k \right\}  \right|$. It is well known that $L^p (\Om) \subset \mathcal M^{p}(\Om)\subset L^q(\Om)$, whenever $1\leq  q<p.$ 
Let us also recall the following consequence of the well known Gagliardo-Nirenberg  inequality (see
\cite[Proposition 3.1]{who}),

\begin{lemma}\label{GN}
Let 
$
\varphi \in L^\infty(0,T, L^2(\Om))\cap L^2(0,T,W^{1,2}_0(\Om))
$ and let $\sigma:=2 \frac {N+2}{N}$.
Then
\[
\int_{\Om_T} |  \varphi |^\sigma \dx x\dx t
\leq
C(N)
\left(
\esssup_{t \in (0,T)} \int_\Om |\varphi(t)|^2\dx x
\right)^{2/N}
\int_{\Om_T} |\nabla \varphi |^2\dx x\dx t.
\]
\end{lemma}
For our purposes it is also useful to recall, for every  $s\in \mathbb R$ and $k>0$, the definition of the following truncation functions
\[
T_k(s)=\max\{-k,\min\{s,k\}\} \quad \mbox{and} \quad G_k(s)=s-T_k(s).
\] 
We shall also use the following property related to differential inequalities.
\begin{lemma}\label{odeode}
Take $K,C\in(0,\infty)$ and assume that $y(t)$ is a positive function in $W^{1,1}_{loc}(0,T)$ such that
\[
y'+Ky^{1+a}\le Cy \quad \mbox{a.e. in } (0,T).
\]
Therefore
\[
y(t)\le \left(\frac{1}{aK}\right)^{1/a}\frac{e^{Ct}}{t^{1/a}}
\]
\end{lemma}
\begin{proof}
See for instance Lemma 2.6 of \cite{Porr2001} or formula (2.18) of \cite{BOP}.
\end{proof}
We conclude this section specifying the notion of solution we shall use through the work.

\begin{defin}\label{defsol1}
If $u_0\in L^{1}(\Omega)$ and $f\in L^{1}(\Omega_T)$, we say that $u\in C([0,T);\elle1)\cap L^1_{loc}(0,T, W^{1,1}_0(\Omega))$ is a weak solution to \eqref{problemintro} if $u(x,0)=u_0(x)$, $|E||u|^{\theta+1}\in L^1(\Omega_T)$, and 
\begin{eqnarray*} 
\io \big(u(t_2)\varphi(t_2)-u(t_1)\varphi(t_1)\big)\dx x-
 \int_{t_1}^{t_2}\int_{\Omega} u\partial_t\varphi\\
+\int_{t_1}^{t_2}\int_{\Omega} \big(M\nabla u+E|u|^{\theta}u \big)\nabla\varphi\dx x \dx t=\int_{t_1}^{t_2}\int_{\Omega}  f  \varphi  \dx x \dx t\\
\end{eqnarray*} 
for every $0<t_1<t_2<T$ and every  $\varphi\in  C^{\infty}_c(\Omega_T)$.
\end{defin}

\begin{remark}\label{helpingremark}
If a solution enjoys the following additional regularity
\[
u\in  W^2_{loc}(0,T) \quad |E||u|^{\theta+1}\in L^2_{loc}((0,T);\elle2),
\]
by the density of $C_c^{\infty}(\Omega_T)$ in $W^2(0,T)$ and a truncation argument, we deduce that 
\[
\int_{t_1}^{t_2}\langle \partial_t u,\varphi \rangle\dx t+\int_{t_1}^{t_2}\int_{\Omega} \big(M\nabla u+E|u|^{\theta}u \big)\nabla\varphi\dx x \dx t=\int_{t_1}^{t_2}\int_{\Omega}  f  \varphi  \dx x \dx t
\]
for every $0<t_1<t_2<T$ and all $\varphi\in W^{2}(0,T)\cap L^{\infty}(\Omega_T)$. Here $\langle \cdot , \cdot \rangle$ is the duality pairing between $W^{-1,2}(\Omega)$ and $W^{1,2}_0(\Omega)$. In the case the additional regularity is global on the time interval $(0,T)$ we can also take $t_1=0$ and $t_2=T$.
\end{remark}

****

\section{Preliminary results}\label{preliminari}

Let us introduce a suitable sequence of approximating solutions to our problem. For any $T>0$ and $n\in \mathbb N$, classical results (see for instance \cite{Li}) imply the existence of $\un\in W^2(0,T)$ solution to
\begin{equation}\label{pn}
\begin{cases}
\partial_t u_n - {\rm div} \,( M(t,x)\nabla u_n+E(t,x)
g_n(u_n))=   f_n(t,x) & \quad  \mbox{in } \Omega_T,\\
u_n= 0 & \quad \mbox{on }  (0,T) \times \partial \Omega ,\\
u_n(0,x)=u_{0,n}(x) & \quad \mbox{in } \Omega.
\end{cases}
\end{equation}
where $f_n=T_n(f)$, $u_{0,n} =T_n(u_0)$, and $g_n(s)=|s|^{\theta}s/(1+\frac 1n|s|^{\theta+1})$. 
 Notice that, thanks to our assumption on $E$, there exist $r>N$ and $s>1$ such that
\[
\frac{N}{2r}+\frac{1}{s}<\frac12 \quad \mbox{and} \quad |E|\in L^s((0,T),\elle r).
\] 
Therefore we also know that $\un\in L^{\infty}(\Omega_T)$ (see \cite{AS}).\\

Our first result states that the $L^1$-norm of $\un$  is uniformly bounded in time.
\begin{prop}\label{stimaL1} Let  $u_0\in L^1(\Omega)$ and $f\in L^1(\Omega_T)$. Then, every $u_n$ solution to \eqref{pn} satisfies
\begin{equation}\label{elleuno}
\io|\un|(t)dx \le M_0:=  \|  u_0 \| _{L^1(\Omega)}+\| f \| _{L^1(\Omega_T)} \quad \mbox{for all } t\in(0,T).
\end{equation}

\end{prop}
\begin{proof}
  Let $t \in (0,T)$ and let $u_n$ be  solution of $(\mathcal P_n)$. We test  \eqref{pn} by $\epsilon^{-1}T_\epsilon(u_n)\chi_{(0,t)}$.  Denoting
\be\label{tunnel}
\Phi_{\epsilon}(u):=\frac{1}{\epsilon}\int_0^u T_{\epsilon}(z)dz
\ee
we have
\begin{equation}\label{14:53bis}
\begin{split}
\int_\Omega  & \Phi_{\epsilon}\left(u_n(t)\right)dx
-
\int_\Omega \Phi_{\epsilon}\left(u_0\right)dx
+\frac{\alpha}{{\epsilon}} \int_{\Omega_t}|\nabla T_{\epsilon}(u_n)|^2dxds
\\
\le&\frac{1}{\epsilon} \int_{\Omega_t}|\un|^{\theta+1} |E||\nabla T_{\epsilon}(u_n)|dxds+\int_{\Omega_t}  |f|   dxds\\
\le&\frac{\alpha}{2{\epsilon}} \int_{\Omega_t}|\nabla T_{\epsilon}(u_n)|^2dxds+ \frac {\epsilon^{2\theta} }\alpha   \int_{\Omega_t}|E|^2dxds+\int_{\Omega_t}  |f|   dxds,
\end{split}
\end{equation}
where we used Young inequality. Dropping the (positive) energy term, we get
\[
\int_\Omega   \Phi_{\epsilon}\left(u_n(t)\right)dx\le \int_\Omega \Phi_{\epsilon}\left(u_0\right)dx+ \frac {\epsilon^{2\theta} }\alpha \int_{\Omega_t}|E|^2dxds+\int_{\Omega_t}  |f|   dxds.
\]
Observing that
\[
\Phi_{\epsilon} (u_n)=\frac{\min\{|u_n|,\epsilon\}^2}{2\epsilon}+\max\{|u_n|,\epsilon\}-\epsilon,
\] 
we can pass to the limit as $\epsilon\to 0$ and conclude that, for a.e. $t\in (0,T)$,
\begin{equation}\label{5}
\begin{split}
\|
u_n(t)
\|_{L^1(\Omega)}
\le \|  u_0 \| _{L^1(\Omega)}+\| f \| _{L^1(\Omega_T)}
\end{split}
\end{equation}
so that \eqref{elleuno} is proved.

\end{proof}

\begin{remark}
We explicitly observe that Proposition \ref{stimaL1} implies that, for any $h\in (0,T)$,
\begin{equation}\label{8}
\|u_n\|_{L^1(\Omega_h)}=\int_0^{ h }\int_{\Omega}|\un(t)|\le  h  M_0.
\end{equation}
which in turn implies the following uniform estimate of the superlevel sets of  $u_n$:
\begin{equation}\label{6}
\begin{split}
|
\{
(t,x)\in\Omega_T\colon |u_n(t,x)|>k
\}
|
\le \frac{T M_0} k
\quad \text{for all $k>0$.}
\end{split}
\end{equation}
\end{remark}

Now we provide a comparison principle for reasonably well behaved solutions. Notice that the only global property we need to assume is $C([0,T]; \elle1)$.
\begin{teo}\label{comparison}
Take $\theta>0$, $f,g\in L^1(\Omega_T)$, and $E\in L^2(\Omega_T)$.
Let $v,w \in C([0,T]; \elle1)\cap  W^2_{loc}(0,T)$ be such that $|E|(|v|^{\theta+1}+|w|^{\theta+1})\in L^2_{loc}((0,T);\elle2)$ and 
\begin{equation*}
\int_{t_1}^{t_2}\langle \partial_t v,\varphi \rangle+\int_{t_1}^{t_2}\int_{\Omega} \big(M\nabla v\nabla\varphi+E|v|^{\theta}v\nabla\varphi \big)\le \int_{t_1}^{t_2}\int_{\Omega}f\varphi ,
\end{equation*}
\begin{equation*}
\int_{t_1}^{t_2}\langle \partial_t w,\varphi \rangle+\int_{t_1}^{t_2}\int_{\Omega} \big(M\nabla w\nabla\varphi + E|w|^{\theta}w\nabla\varphi\big) \ge \int_{t_1}^{t_2}\int_{\Omega}g\varphi,
\end{equation*}
for every $0<t_1<t_2<T$ and every $\varphi\in W^2(0,T)\cap L^{\infty}(\Omega_T) $, $\varphi\geq 0$ in $\Om_T$.
Assume moreover that 
\[
 (|v|^{\theta}+|w|^{\theta})|E|\in L^2(\Omega_T).
\]
Therefore, it follows that for any $\tau \in[0,T]$
\[
\int_{\Omega}(v(\tau)-w(\tau))_+\le\int_0^{\tau}\int_{\Omega}(f-g)\chi_{v>w}+\int_{\Omega}(v(0)-w(0))_+
\]
\end{teo}
\begin{proof}
We fix $0<\delta<\tau<T$, $z=(v-w)_+$ and use 
 $\epsilon^{-1}T_\varepsilon (z)$ as a test function in the inequalities satisfied by $v$ and $w$. Recalling \eqref{tunnel}, we get
\begin{equation*}
\begin{split}
&\int_{\Omega}\Phi_{\epsilon}(z)(\tau)dx
+\frac{\alpha}{\epsilon}\int_{\delta}^{\tau}\int_{\Omega} |\nabla T_\varepsilon (z)|^2dxdt \\ 
&\le \frac{1}{\epsilon} \int_{\delta}^{\tau}\int_{\Omega} |E| \left||v|^{\theta}v -|w|^{\theta}w\right| |\nabla T_\varepsilon (z)|dxdt
+ \frac1{\epsilon}\int_{\delta}^{\tau}\int_{\Omega}(f-g)T_\varepsilon (z)dxdt+\int_{\Omega}\Phi_{\epsilon}(z)(\delta)dx
\end{split}
\end{equation*}
Since $T_{\epsilon}z=T_\varepsilon(v-w)_+$ is supported in $\{0<z<\epsilon\}$, it follows that
\begin{equation*}
\begin{split}
\int_{\Omega}\Phi_{\epsilon}(z)(\tau)dx
&+\frac{\alpha}{2\epsilon}\int_{\delta}^{\tau}\int_{\Omega} |\nabla T_\varepsilon (z)|^2dxdt \\ 
&\le \frac{1}{2\alpha\epsilon}\int_{\delta}^{\tau}\int_{\Omega} |E|^2 \left||v|^{\theta}v -|w|^{\theta}w)\right|^2 \chi_{0<z<\epsilon}dxdt\\
&+ \frac1{\epsilon}\int_{\delta}^{\tau}\int_{\Omega}(f-g)T_\varepsilon (z)dxdt+\int_{\Omega}\Phi_{\epsilon}(z)(\delta)dx
\end{split}
\end{equation*}
As $\left||v|^{\theta}v -|w|^{\theta}w)\right|\le c(|v|^{\theta}+|w|^{\theta})|v-w|$,
the above inequality becomes
\begin{equation*}
\begin{split}
\int_{\Omega}\Phi_{\epsilon}(z)(\tau)dx
\le& \tilde c \epsilon \int_{\delta}^{\tau}\int_{\Omega} |E|^2 \left(|v|^{2 \theta }+|w|^{2 \theta }\right)\chi_{0<z<\epsilon} dxdt\\
+& \frac1{\epsilon} \int_{\delta}^{\tau}\int_{\Omega}(f-g)T_\varepsilon (z)dxdt+\int_{\Omega}\Phi_{\epsilon}(z)(\delta)dx
\end{split}
\end{equation*}
Therefore, taking at first the limit as $\epsilon\to 0$, it follows that
\[
\int_{\Omega}(v(\tau)-w(\tau))_+\le\int_{\delta}^{\tau}\int_{\Omega}(f-g)\chi_{v>w}+\int_{\Omega}(v(\delta)-w(\delta))_+.
\]
Finally, thanks to the continuity assumption $C([0,T]; \elle1)$, we let $\delta\to0$ and obtain the desired result. 
\end{proof}

\section{Proof of Theorems \ref{thetasmalf=0} and \ref{thetabigf=0}} \label{differentialineq}

In the present section we provide a proof of Theorem \ref{thetasmalf=0} and Theorem \ref{thetabigf=0}. 
As already said in the introduction, we shall obtain a differential inequality satisfied by $y(t)=\io |u_n(t)|^m$. The first step to deduce such an inequality formally consists in testing \eqref{pn} with $|u_n|^{m-2}u_n$ with $m>1$. However, such a choice of test function is not allowed if $m\in(1,2)$. Moreover, we also need to now that $y(t)$ admits (weak) derivative. The next Lemma addresses these technical issues (here we adapt to our framework Lemma 2.5 of \cite{Porr2001}).

\begin{lemma}\label{19:03} Take a function $v\in W^{2}_{loc}(0,T)$ such that $|E||v|^{\theta+1}\in L^2_{loc}((0,T);\elle2)$ and that
\[
\int_{t_1}^{t_2}\langle \partial_t v,\varphi \rangle\dx t+\int_{t_1}^{t_2}\int_{\Omega} \big(M\nabla v+E\tilde g(v) \big)\nabla\varphi\dx x \dx t=0,
\]
for every $0<t_1<t_2<T$ and all $\varphi\in L^{2}_{loc}(0,T, W^{1,2}_0(\Omega))\cap L^{\infty}(\Omega_T)$, where $|\tilde g(v)|\le |v|^{\theta+1
}$. If, for a given $m>1$,
\be\label{19:03}
v\in L^{\infty}_{loc}((0,T);\elle{m}) \quad \mbox{ and } \quad |E|^2|v|^{2 \theta +m}\in L^1_{loc}((0,T),\elle1),
\ee
it follows that $y(t)=\io |v(t)|^m$ belongs to $W^{1,1}_{loc}(0,T)$ and verifies
\be\label{testcontr}
\frac{d}{dt}\io  | v(t)|^{ m}+\alpha\frac{2\mathcal{S}^2 }{m'}\left( \io  | v(t)|^{m\frac{2^*}2}\right)^{\frac2{2^*}}\le m\frac{m-1}{2\alpha}\io |E|^2 | v(t)|^{2 \theta +m}.
\ee

\end{lemma}
\begin{proof}
For $\eps>0$, we take the globally Lipschitz function $\psi_\varepsilon(s)$ and its integral function $\Psi_\varepsilon(s)$:
\[
\psi_\varepsilon(s):=\frac1{m-1}\left[ (T_{1/\eps}(|s|)+\varepsilon)^{m-1} - \varepsilon^{m-1} \right]{\rm sign}(s) \quad \mbox{and}\quad \Psi_\varepsilon(s):=\int_0^s \psi_\varepsilon(v)dv.
\]
Therefore we can chose $\psi_\varepsilon ( v) \chi_{(t_1,t_2)}$, with $0<t_1<t_2<T$, as a test function in \eqref{pn}. Notice that
\be\label{proprigra}
\nabla\psi_\varepsilon ( v)=(T_{1/\eps}(|v|)+\varepsilon)^{m-2}\nabla v \, \chi_{|v|<1/\eps} =(|v|+\varepsilon)^{m-2}\nabla T_{1/\eps}( v).
\ee
Therefore, it follows that
\[
\begin{split}
&\quad \int_\Omega 	 \Psi_\varepsilon( v(t_2))dx + \alpha \int_{t_1}^{t_2} \int_{\Omega} |\nabla  T_{1/\eps}(v)|^2  (| v|+\varepsilon)^{m-2} dx dt
\\
& \le
 \int_{t_1}^{t_2} \int_{\Omega} |E|| v|^{\theta+1} |\nabla  T_{1/\eps}(v)|  (| v|+\varepsilon)^{m-2} dx dt
+ \int _ \Omega \Psi_\varepsilon( v(t_1))dx.
\end{split}
\]
Using Young inequality in the drift term and recalling \eqref{proprigra}, it follows that
\begin{equation}\label{stima2}
\begin{split}
& \quad \int_\Omega 	 \Psi_\varepsilon( v(t_2))dx + \frac{\alpha}{2}\int_{t_1}^{t_2} \int_{\Omega} |\nabla  T_{1/\eps}(v)|^2  (| v|+\varepsilon)^{m-2} dx dt\\
& \le \frac{1}{2\alpha} \int_{t_1}^{t_2} \int_{\Omega} |E|^2(T_{1/\eps}(|v|)+\epsilon)^{2 \theta +m}dx dt  + \int _ \Omega \Psi_\varepsilon(v(t_1))dx .
\end{split}
\end{equation}
Thanks to assumption \eqref{19:03} and since $|\Psi_\varepsilon(v)|\le C(m)(|v|^m +1)$, it follows that the right hand side above is bounded for almost all $t_1\in(0,t_2)$.
Therefore, taking the limsup in \eqref{stima2} as $\eps\to0$,  Fatou's Lemma implies that $|\nabla  v|^2  | v|^{m-2}\in L^1_{loc}((0,T),\elle1)$.\\
On the other hand, taking $\psi_{\eps}( v)\zeta$ as a test function in \eqref{pn}, with $\zeta\in C^{\infty}_c(0,T)$, one deduces that
\be\label{11:31}
\frac{d}{dt}\io \Psi_{\eps}( v)dx =-\io \big(M\nabla v+ E\tilde g(v)\big)\cdot\nabla T_{1/\eps}( v)(|v|+\varepsilon)^{m-2} dx 
\ee
in the sense of distributions. Since we know that $|\nabla  v|^2  | v|^{m-2}, |\nabla v|^2\in L^1_{loc}((0,T),\elle1)$ and thanks to assumptions \eqref{19:03}, we deduce that the right hand side of \eqref{11:31} (that is function of $t$) is controlled by an element in $L^1_{loc}(0,T)$ uniformly as $\eps\to0$. Therefore, the dominated converge theorem implies that the left hand side of \eqref{11:31} strongly converges in $L^1(0,T)$ as $\eps\to0$. Since $\Psi_{\eps}(s)\to \frac{|s|^m}{m(m-1)}$, this gives that $y(t)=\io| v(t)|^m$ belongs to $W^{1,1}_{loc}(0,T)$. Going back to \eqref{11:31} and using Young and Sobolev inequalities, we deduce that
\be\label{11:31bis}
\begin{split}
\frac{d}{dt}\io \Psi_{\eps}( v)dx +&\alpha\mathcal{S}^2\frac{2}{m^2}\left(\io|( T_{1/\eps}(|v|)+\eps)^{\frac m2}-\eps^{\frac m2}|^{2^*}dx\right)^{\frac{2}{2^*}}\\
\le& \frac{1}{2\alpha}\io|E|^2(T_{1/\eps}(|v|)+\epsilon)^{2 \theta +m}dx.
\end{split}
\ee
Taking the limit as $\eps\to0$ we get \eqref{testcontr}.
\end{proof}

Now we are ready to prove the following proposition.

\begin{prop}\label{propthetasmalf=0}
Let $\un$ be sequence of the approximating problems \eqref{pn}, assume that $f\equiv0$, $\theta\le 1/N$, $r\in(N,\infty]$, and that \eqref{thetasmallbis} holds true.
If $E\in L^{\infty}((0,T);L^r(\Omega))$, and $u_0\in \elle{\mu}$ with $\mu\in[1,\infty)$, it follows that
\be\label{pre17:03}
\|\un(t)\|_{\elle \mu}\le \|u_0\|_{\elle \mu} e^{C_1 t} \qquad \mbox{ for }t\in (0,T),
\ee
with $C_1=C_1(\alpha, N, \theta, \mu, r,|E|, \|u_0\|_{\elle 1})$. Moreover for any given $m>\mu$
\begin{equation}\label{17:03}
\|u_n(t)\|_{L^m (\Omega )}\leq  \|u_0\|_{L^{\mu}(\Omega)}\frac{ e^{C_2 (T+1)}}{t^{\frac N2 (\frac{1}{\mu}-\frac 1m)} },\qquad \mbox{ for }t\in (0,T),
\end{equation}
with $C_2=C_2(\alpha, N, \theta, m, \mu, r,|E|, \|u_0\|_{\elle1})$.
\end{prop}
\begin{remark}
Notice that if $\mu=1$, we already know from Proposition \ref{stimaL1} that $\|\un(t)\|_{\elle1}\le \|u_0\|_{\elle1}$, namely, estimate \eqref{pre17:03} holds true with $C_1=0$.
\end{remark}
\begin{proof}The first ingredient we need is a decay estimate on the measure of a suitable super level set. For any $\rho>0$, we have that
\begin{equation*}
\begin{split}
&|\{x\in \Omega \ : \ |E||\un(t)|^{\theta}>\rho\}|\\ = &|\{x\in \Omega \ : \ |E||\un(t)|^{\theta}>\rho, |E|\ge \rho^{ 1/2}\}|+ |\{x\in \Omega \ : \ |E||\un(t)|^{\theta}>\rho, |E|< \rho^{ 1/2}\}|\\
\le& |\{x\in \Omega \ :\ |E|\ge \rho^{ 1/2}\}|+ |\{x\in \Omega \ : \ |\un(t)|^{\theta}>\rho^{1/2}\}|\\
\le & \|E\|_{L^{\infty}(0,T,L^{ 2}(\Omega))}^{2} \frac{1}{\rho}+\|u_0\|_{L^{1}(\Omega)}\frac{1}{\rho^{\frac{1}{2\theta}}},
\end{split}
\end{equation*}
with the right hand side above vanishing as $\rho$ tends to $+\infty$ uniformly with respect to $t$. Moreover, since $\un\in L^{\infty}(\Omega_T)$, we can apply Lemma \ref{19:03} with $v=\un$ and $\tilde g= g_n$ (the one used in \eqref{pn}). Therefore we have, for any $m>1$,
\[
\frac{d}{dt}\io  | \un(t)|^{ m}+\alpha\frac{2\mathcal{S}^2 }{m'}\left( \io  | \un(t)|^{m\frac{2^*}2}\right)^{\frac2{2^*}}\le m\frac{m-1}{2\alpha}\io |E|^2 | \un(t)|^{2 \theta +m}.
\]
Let us decompose the first integral on the right hand side above as follows
\be\label{29-4}
\begin{split}
& \quad \io |E|^2 |\un(t)|^{2 \theta } |\un(t)|^{ m}\\
 &\le \rho^2\int_{\{|E| |\un(t)|^{\theta}\le \rho\}} |\un(t)|^{ m}+\int_{\{|E| |\un(t)|^{\theta}> \rho\}}|E|^2 |\un(t)|^{2\theta} |\un(t)|^{ m}\\
& \le\rho^2\io |\un(t)|^{ m}+\left(\int_{\{|E| |\un(t)|^{\theta}> \rho\}}|E|^{ r}\right)^{\frac{2}{ r}}   \|u_0\|_{\elle{1}}^{2 \theta }\|\un^{\frac{ m}{2}}\|_{\elle{2^*}}^{2} 
\end{split}
\ee
Let us explicitly stress that in the last line we used $\theta\le 1/N$ and assumption \eqref{thetasmallbis} in order to apply  H\"older inequality. Thanks to the decay estimate on $|\{|E| |\un(t)|^{\theta}> \rho\}|$, we chose $\overline\rho=\overline{\rho}(\alpha, N, \theta, \mu, r,|E|, \|u_0\|_{\elle 1})$ large enough to conclude that
\be \label{19:05}
\frac{d}{dt}\io  |\un(t)|^{ m}+\alpha\frac{\mathcal{S}^2 }{m'}\left( \io  |\un(t)|^{m\frac{2^*}2}\right)^{\frac2{2^*}}\le m \frac{m-1}{2\alpha} \overline{\rho}^2  \io |\un(t)|^{ m}.
\ee

Let us now prove estimate \eqref{pre17:03}.  If $\mu=1$ the result follows by directly by Lemma \ref{stimaL1} (with $f\equiv0$) and with $C_1=0$. If $\mu>1$, we take $m\equiv \mu$ in \eqref{19:05}, dropp the positive term on the left hand side, and integrate. Therefore, we get
\[
\| u_n(t)\|_{\elle \mu}\le \| u_0\|_{\elle \mu} e^{C_1 t},
\]
with $C_1=\frac{(\mu-1)}{2\alpha} \overline{\rho}^2$. 

To prove estimate \eqref{17:03}, let us take $m>\mu$. Using the interpolation inequality between $\mu$ and $m2^*/2$, it follows that 
\[
\begin{split}
\left( \io  |\un(t)|^{m}\right)^{1+\frac{2\mu}{N(m-\mu)}}&\le\left( \io  |\un(t)|^{\mu}\right)^{\frac{2m}{N(m-\mu)}}\left( \io  |\un(t)|^{m\frac{2^*}2}\right)^{\frac2{2^*}}\\
&\le \left(  \| u_0\|_{\elle \mu}    e^{C_1 T}     \right)^{\frac{2m\mu}{N(m-\mu)}}\left( \io  |\un(t)|^{m\frac{2^*}2}\right)^{\frac2{2^*}}.
\end{split}
\]
Plugging such an estimate into \eqref{19:05}, we deduce that the differential inequality satisfied by $y(t)=\io |\un(t)|^m$ becomes
\[
y'+K y^{1+a}\le C_1 y
\]
with $a=\frac{2\mu}{N(m-\mu)}$, $K=\alpha\frac{\mathcal{S}^2 }{m'}\left(  \| u_0\|_{\elle \mu}    e^{C_1 T}     \right)^{-\frac{2m\mu}{N(m-\mu)}}$ and $C=  C_1$. We can now apply Lemma \ref{odeode} to conclude that
\[
y(t)\le\left(\frac{1}{Ka}\right)^{1/a}\frac{e^{C_1t}}{t^{1/a}}= \left(\frac{m'}{\alpha\mathcal{S}^2 a}\right)^{1/a} \| u_0\|_{\elle \mu}^m \frac{e^{C_1(1+m)T}}{t^{\frac{N(m-\mu)}{2\mu}}}.
\]

\end{proof}
Let us now provide the proof of Theorem  \ref{thetasmalf=0}.

\begin{proof}[Proof of Theorem \ref{thetasmalf=0}] \textbf{Step 1}. Let us assume at first $\mu\ge 2$. Therefore, estimate \eqref{pre17:03} implies that
\[
\|u_n(t)\|_{\elle{2}}\le  \|u_0\|_{\elle{2}}  e^{C_1T} \quad t\in(0,T).
\]
Taking $\un \chi_{(0,T)}$ as a test function in \eqref{pn}, it follows that
\[
\io \un^2(T)dx+\frac{\alpha}{2}\int_{0}^{T}\io|\nabla \un|^2 dxdt \le \int_{0}^{T}\io |E|^2|\un|^{2(\theta+1)}dxdt+\io u_0^2dx
\]
Using estimate \eqref{29-4} with $m=2$, it follow that
\[
\int_{\Omega}|\un(t)|^{2(\theta+1)}|E|^2dx\le\rho^2\io |\un(t)|^{2}+\left(\int_{\{|E| |\un(t)|^{\theta}> \rho\}}|E|^{ r}\right)^{\frac{2}{ r}}   \|u_0\|_{\elle{1}}^{2 \theta }\|\nabla \un\|_{\elle{2}}^{2}.
\]
Choosing, $\rho=\rho(E, u_0,\theta,\alpha)$ large enough, we have that
\be\label{infondo}
\int_{0}^{T}\io |\un|^{2(\theta+1)}|E|^2dxdt \le T \rho^2 \|u_0\|_{\elle{2}}e^{C_1T}+\frac{\alpha}{2}\int_{0}^{T}\io|\nabla \un|^2 dxdt.
\ee
Therefore, we conclude that the sequence $\un$ is bounded in $L^2((0,T), W^{1,2}_0(\Omega))$ and that the sequence $|\un|^{\theta+1}|E|$ is bounded in $L^2(\Omega_T)$. Going back to the equation solved by $\un$, we deduce that $\partial_t u_n$ is bounded in $L^2((0,T), W^{-1,2}(\Omega))$. This imply that there exists $u\in W^2(0,T)$ such that, up to subsequences, $\un\rightharpoonup u$ in $W^2(0,T)$. Standard compactness results also assure that $\un\to u$ strongly in $L^2(\Omega_T)$ and $a.e.$ and that $u\in C([0,T], L^2(\Omega))$.  Therefore we can pass to the limit in \eqref{pn} and in the estimates \eqref{pre17:03}-\eqref{17:03}, as $n\to \infty$, to obtain a solution of \eqref{introintro} with the desired properties. 

\textbf{Step 2}. To deal now with the case $\mu\in[1,2)$, let us consider the sequence $v_n\in W^2(0,T)$ of solutions to
\be\label{9:35}
\int_0^T\langle \partial_t v_n ,\varphi \rangle dt
+
 \int_{\Omega_T} \big(M(t,x) \nabla v_n+ |v_n|^{\theta}v_n E\big) \cdot \nabla \varphi \dx x\dx t
=  0, 
\ee
with $v_n(0,x)=T_n(u_0)$ and $\varphi\in L^{2}(0,T, W^{1,2}_0(\Omega))\cap L^{\infty}(\Omega_T)$. The existence of such a family of solutions is assured by Step 1 (see also Remark \ref{helpingremark}).  Applying the comparison principle provided by Theorem \ref{comparison} to $v_n$ and $v_m$, we get that
\[
\int_{\Omega}|v_n(t,x)-v_m(t,x)|dx\le \int_{\Omega}|T_n(u_0)(x)-T_m(u_0)(x)|dx,
\]
that implies that the sequence $v_n$ is Cauchy in $C([0,T], L^1(\Omega))$.\\ Now we claim that we can apply Lemma \ref{19:03} with the choice $v=v_n$ and $\tilde g(s)=|s|^{\theta}s$. Clearly we have $v_n\in W^{2}_{loc}(0,T)$. Moreover, thanks again to Step 1, we know that for any $q>2$
\[
\|v_n(t)\|_{L^q (\Omega )}\leq  C(T)\|u_0\|_{\elle2}t^{-\frac N2 (\frac{1}{2}-\frac 1q)}\qquad \mbox{ for } t\in (0,T).
\]
However since $\Omega$ is bounded this implies that
\[
v_n\in L^{\infty}_{loc}((0,T), \elle q) \quad \forall q\ge1.
\]
This concludes the proof of the claim. Therefore, we have that for any $m>1$, $v_n$ satisfies
\[
\frac{d}{dt}\io  | v_n(t)|^{ m}+\alpha\frac{2\mathcal{S}^2 }{m'}\left( \io  | v_n(t)|^{m\frac{2^*}2}\right)^{\frac2{2^*}}\le m\frac{m-1}{2\alpha}\io |E|^2 | v_n(t)|^{2 \theta +m}.
\]
Following the very same procedure of Proposition \ref{propthetasmalf=0} (we omit further details for the sake of brevity) we deduce that, for any $m\ge\mu$,
\be\label{9:34}
 \|v_n(t)\|_{L^m (\Omega )} \le  C \|u_0\|_{L^\mu (\Omega )}t^{-\frac N2 \big(\frac1\mu-\frac 1m\big)} \qquad \mbox{ for }t\in (0,T).
 \ee
Therefore, taking $v_n \chi_{(\epsilon,T)}$ with $\epsilon>0$ as a test function in \eqref{9:35}, it follows that
\[
\io v_n^2(T)dx+\frac{\alpha}{2}\int_{\epsilon}^{T}\io|\nabla v_n|^2 dxdt \le \int_{\epsilon}^{T}\io |v_n|^{2(\theta+1)}|E|^2|dxdt+\io v_n(\epsilon)^2dx.
\]
Taking $m=2^*>2>\mu$ in \eqref{9:34}, we deduce that
\[
\|v_n(t)\|_{L^{2^*} (\Omega )}\leq C \|u_0\|_{L^{\mu}(\Omega)}  \epsilon^{\frac N2 (\frac 1{2^*}-\frac{1}{\mu})},
\]
that implies
\be\label{infondobis}
\begin{split}
\int_\epsilon^{T}\int_{\Omega}|v_n(t)|^{2(\theta+1)}|E|^2dxdt&\le\int_0^{T} \|v_n(t)\|_{\elle{2^*}}^2 \|v_n(t)\|_{\elle1}^{2(\theta+1)-2}\|E(t)\|_{\elle r}^2dt\\
&\le C(\epsilon,T) \|u_0\|_{L^{\mu}(\Omega)}^2\|u_0\|_{\elle1}^{2(\theta+1)-2}\|E\|_{L^{\infty}((0,T),\elle r)}^2,
\end{split}
\ee
with the constant $C(\epsilon,T)$ that diverges as $\epsilon\to0$. All this information says that $|E||v_n|^{\theta+1}$ is bounded in $L^2_{loc}(0,T, \elle2)$, $v_n$ is bounded in $L^2_{loc}(0,T, W^{1,2}_0(\Omega))$, and that  $\partial_t v_n$ is bounded in $L^2_{loc}(0,T, W^{-1,2}(\Omega))$. Therefore, there exists $u\in W^2_{loc}(0,T)\cap C([0,T],\elle 1))$  limit of the sequence $v_n$. Passing to the limit as $n\to \infty$ in \eqref{pn}, we conclude that $u$ is a distributional solution of our problem. Applying Fatou's Lemma in \eqref{9:34} provides us also the required estimate on $\|u(t)\|_{\elle m}$.

In both Steps, uniqueness follows by the comparison principle (Theorem \ref{comparison}), the fact that $u\in  W^2_{loc}(0,T)$, $|E||u|^{\theta+1}\in L^2_{loc}(0,T,\elle2)$, and Remark \ref{helpingremark}.
\end{proof}

\begin{prop}\label{12:15prop}
 Take $\theta>0$,  $r\in(N,\infty]$, and $\mu\in[1,\infty)$ such that assumption \eqref{thetabig} holds true.
Assume that $f\equiv 0$, $E\in L^{\infty}((0,\infty);L^r(\Omega))$, and $u_0\in \elle{\mu}$.
Then, there exists $T^{*}>0$ that depends on $\alpha, N,\mu, r, \theta, \|E\|_{L^{\infty}(0,\infty,L^{ r}(\Omega))}, \|u_0\|_{\elle{\mu}}$ such that, if $\un$ are the solution to \eqref{pn} with $T= T^*$, it follows that
\be\label{13:35}
\|\un(t)\|_{\elle{\mu}}\le \|u_0\|_{\elle{\mu}}\left(\frac{T^*}{T^*-t }\right)^{\gamma} \qquad \mbox{ for }t\in (0,T^*),
\ee
for some positive exponent $\gamma=\gamma(\mu, r, N)$. Moreover for any given $m>\mu$
\be\label{13:35bis}
 \|u_n(t)\|_{L^m (\Omega )} \le  \|u_0\|_{L^{\mu}(\Omega)}\frac{C}{t^{\frac N2 \big(\frac1\mu-\frac 1m\big)}}h(t) \qquad \mbox{ for }t\in (0,T^*),
\ee
with $C=C(\alpha, N, \theta, m, \mu, r, \|E\|_{L^{\infty}(0,\infty,L^{ r}(\Omega))})$ and  $h\in C([0,T^*))$ an increasing function such that $h(0)=1$ and $\lim_{t\to T^*} h(t)=+\infty$.
\end{prop}

\begin{proof} 
As before, the fact that $\un\in L^{\infty}(\Omega_T^*)$ allows us to apply Lemma \ref{19:03} with $v=\un$ and deduce that, for any $m>1$,
\be\label{10:00}
\frac{d}{dt}\io  | \un(t)|^{ m}+\alpha\frac{2\mathcal{S}^2 }{m'}\left( \io  | \un(t)|^{m\frac{2^*}2}\right)^{\frac2{2^*}}\le m\frac{m-1}{2\alpha}\io |E|^2 | \un(t)|^{2 \theta +m}.
\ee
Notice that for any $m\ge \mu$ we have that
\[
m<(2(\theta+1)-2+m)\frac{r}{r-2}<\frac{Nm}{N-2},
\] 
where the first inequality follows from the facts that $r>N$ and $\theta>0$, while the second one is a consequence of assumption \eqref{thetabig}. Therefore, it follows that
\[
\begin{split}
&\quad \int_{\Omega} |E|^2 |\un(t)|^{2 \theta +m}\\
&\le \|E(t)\|_{L^{r}(\Omega)}^{2}\left(\int_{\Omega} |\un(t)|^{(2(\theta+1)-2+m)\frac{r}{r-2}}\right)^{1-\frac2r}\\
&\le \|E(t)\|_{L^{r}(\Omega)}^{2} \left(\int_{\Omega} |\un(t)|^{m}\right)^{\frac{rm-Nm-r \theta (N-2)}{rm}}\left(\int_{\omega} |\un(t)|^{\frac{mN}{N-2}}\right)^{\left(\frac{\theta}{m}+\frac1r\right)(N-2)}\\
&\le C_{m} \left(\int_{\Omega} |\un(t)|^{m}\right)^{1+\frac {2r \theta }{mr-mN-Nr \theta }}+\alpha\frac{\mathcal{S}^2 }{m'}\left( \int_{\Omega} |\un(t)|^{m\frac{2^*}2}\right)^{\frac2{2^*}}
\end{split}
\]
with $C_{m}=c(\alpha,N,m)\|E(t)\|_{L^r(\Omega)}^{\frac{2mr}{mr-mN-rN \theta }}$, where we have use H\"older, interpolation, and Young (thanks to assumption \eqref{thetabig} again) inequality in the second, third, and fourth line respectively. 
Plugging this information into \eqref{10:00} leads us to
\begin{equation}\label{13:09}
\frac{d}{dt}\io  |\un(t)|^{ m}+\alpha\frac{\mathcal{S}^2 }{m'}\left( \io  |\un(t)|^{m\frac{2^*}2}\right)^{\frac2{2^*}}
\le C_m \left(\int_{\Omega} |\un(t)|^{m}\right)^{1+\frac {2r \theta }{mr-mN-Nr \theta }}
\end{equation}

To prove \eqref{13:35}, we set $m\equiv \mu$, neglect the second term on the left hand side, and deduce that $y(t)=\io |\un(t)|^{\mu}$ solves
\[
y'\le C_\mu y^{1+b},
\]
with $b=\frac {2r \theta }{\mu r-\mu N-Nr \theta }$. Integration by separation of variables implies that
\[
\frac{1}{y(0)^b}-\frac{1}{y(t)^b}\le \frac{C_{\mu}}{b} t \quad \mbox{for } t\in(0,T^*)
\]
with $T^*=b(C_{\mu} \|u_0\|_{\elle{\mu}}^{b\mu})^{-1}$. Writing such an inequality in terms of the $L^{\mu}$ norm of $\un$ gives us
\[
\|\un(t)\|_{\elle{\mu}}\le \|u_0\|_{\elle{\mu}}\left(\frac{T^*}{T^*-t }\right)^{\frac1{b\mu}} \quad \mbox{for } t\in(0,T^*).
\]

To prove \eqref{13:35bis} we go back to \eqref{13:09} and use once more time interpolation inequality on its second term, to infer that, for $0<t<\tau<T^*$,
\[
\begin{split}
\left( \io |\un(t)|^{m}\right)^{1+\frac{2\mu}{N(m-\mu)}}&\le\left( \io |\un(t)|^{\mu}\right)^{\frac{2m}{N(m-\mu)}}\left( \io |\un(t)|^{m\frac{2^*}2}\right)^{\frac2{2^*}}\\
&\le\left[ \|u_0\|_{\elle \mu}^{\mu} \left(\frac{T^*}{T^*-\tau }\right)^{\frac1{b}}\right]^{\frac{2m}{N(m-\mu)}}\left( \io |\un(t)|^{m\frac{2^*}2}\right)^{\frac2{2^*}}.
\end{split}
\]
Therefore estimate \eqref{13:09} becomes
\[
\frac{d}{dt}\io  |\un(t)|^{ m}+K \left(\int_{\Omega} |\un(t)|^{m}\right)^{1+a}\le C_m \left(\int_{\Omega} |\un(t)|^{m}\right)^{1+d}
\]
with
\[
 a=\frac{2\mu}{N(m-\mu)}, \ \ \   d=\frac {2r \theta }{mr-mN-Nr \theta }, \ \ \  K=\alpha\frac{\mathcal{S}^2 }{m'}\left[ \|u_0\|_{\elle \mu}^{\mu} \left(\frac{T^*}{T^*-\tau }\right)^{\frac1{b}}\right]^{-\frac{2m}{N(m-\mu)}}
\]
Notice that thanks to the strict inequality of assumption \eqref{thetabig} it follows that $a>d$. 
 Setting $y(t)=\int |\un (t)|^m$ we get
\[
y'\le C_m y^{1+d}-K y^{1+a}\le Qy-\frac{K}{2}y^{1+a},
\]
with $Q=(2/K)^{\frac{d}{a-d}}C_m^{\frac{a}{a-d}}(d/a)^{\frac{d}{a-d}}$. The second inequality above follows by the fact that $C_ms^{d}-\frac{K}{2}s^{a}\le Q$ for all $s\ge0$, that can be directly checked. Therefore, we can apply Lemma \ref{odeode} again to conclude that
\[
y(\tau)\le\left(\frac{1}{Ka}\right)^{1/a}\frac{e^{Q\tau}}{\tau^{1/a}}=\left(\frac{m'}{\alpha\mathcal{S}^2 a}\right)^{1/a} \| u_0\|_{\elle \mu}^m \frac{e^{Q \tau}}{\tau^{\frac{N(m-\mu)}{2\mu}}}\left(\frac{T^*}{T^*-\tau }\right)^{\frac{m}{b\mu}}.
\]

\end{proof}

With the previous Proposition at hand, we can now address the proof of Theorem \ref{thetabigf=0}

\begin{proof}[Proof of Theorem \ref{thetabigf=0}] \textbf{Step 1} Consider the sequence $\un$ of solution to \eqref{pn} with $T= T^*$ given by Proposition \ref{12:15prop}. Assume at first that $\mu\ge 2$. 
Taking $|\un|^{\mu-2}\un \chi_{(0,\tau)}$, with $\tau\in(0,T^*)$, as a test funcion in \eqref{pn}, we deduce that
\be\label{15:53}
\begin{split}
 &  \frac{\alpha}{2}\int_{0}^{\tau}\io|\nabla \un|^2|\un|^{\mu-2} dxdt\\  \le& \frac{1}{2\alpha}\int_{0}^{\tau}\io |\un|^{2 \theta }|\un|^{\mu}|E|^2dxdt+\io u_0^{\mu}dx \\
\le&\frac{c(|\Omega|)}{2\alpha}  \int_0^{\tau}\|\un(t)\|_{\elle \mu}^{2(\theta+1)-2}\|\un(t)^{\mu}\|_{\elle{2^*/2}}\|E(t)\|_{\elle r}^2dt+\io u_0^{\mu}dx
\end{split}
\ee
where in the second line we used H\"older inequality and assumption \eqref{thetabig}.
In order to estimate the integral right hand side above, let us recall at first estimate \eqref{13:35}
\[
\|\un(t)\|_{\elle{\mu}}\le \|u_0\|_{\elle{\mu}}\left(\frac{T^*}{T^*-t }\right)^{\gamma} \qquad \mbox{ for }t\in (0,T).
\]
Moreover, integrating formula \eqref{13:09}, between $0$ and $\tau$ with $m=\mu$,  gives that
\[
\begin{split}
\int_0^{\tau}\left( \io  |\un|^{\mu\frac{2^*}2}dx \right)^{\frac2{2^*}}dt
&\le C_\mu \int_0^{\tau} \left(\int_{\Omega} |\un(t)|^{\mu}dx\right)^{1+a}dt+\|u_0\|_{\elle \mu}^{\mu}\\
&\le C\|u_0\|_{\elle{\mu}}^{1+a}\left(\frac{T^*}{T^*-\tau}\right)^{\gamma_1}+\|u_0\|_{\elle \mu}^{\mu}
\end{split}
\]
for suitable exponents $a,\gamma_1$ that depend on $\mu,r, \theta, N$. Therefore estimate \eqref{15:53} becomes (recall that $\mu\ge2$)
\[
\int_{\Omega_{\tau}\cap\{|\un|\ge1\}}|\nabla  \un|^2 dxdt\le\int_{\Omega_{\tau}}|\nabla \un|^2|\un|^{\mu-2} dxdt \le C(\tau).
\]
with $C(\tau)\to\infty$ as $\tau\to T^*$.
On the other hand, if we go back to estimate \eqref{14:53bis}, take $\eps=1$, and recall that here we are assuming $f\equiv0$, it follows that
\[
\int_{\Omega_{\tau}\cap\{|\un|\le1\}}|\nabla  \un|^2 dxdt\le c(\alpha)\big( \|E\|_{L^2(\Omega_T)}+ \|u_0\|_{\elle 1}\big).
\]
These last two estimates allow us to conclude that the sequence $\un$ is bounded in $L^2_{loc}([0,T^*), W^{1,2}_0(\Omega))$. Notice moreover that
\[
\begin{split}
\int_0^{\tau}\io |\un|^{2(\theta+1)}|E|^2\le \|E\|_{L^2(\Omega_{\tau})}^2+\int_{\Omega_{\tau}\cap\{|\un|\ge1\}}|\un|^{2 \theta }|\un|^{2}|E|^2\\
\le \|E\|_{L^2(\Omega_{\tau})}^2+\int_{\Omega_{\tau}\cap\{|\un|\ge1\}}|\un|^{2 \theta }|\un|^{\mu}|E|^2,
\end{split}
\]
with the second integral in the right hand side that can be estimated as at the beginning of the proof.  This means that the sequence $|\un|^{\theta+1}|E|$ is bounded in $L^2_{loc}([0,T^*),L^2(\Omega))$. Therefore, we conclude that also $\partial_t \un$ is bounded in $L^2_{loc}([0,T^*), W^{-1,2}(\Omega))$ and, therefore, $\un$ is bounded in $W_{loc}([0,T^*))$. Summing up, we deduce that there exists $u\in L^2_{loc}([0,T^*), W^{1,2}_0(\Omega))$ such that $\nabla \un \rightharpoonup \nabla u$ in $ L^2_{loc}([0,T^*),(L^2(\Omega))^N)$ and $\un\to u$ in $C([0,T^*),L^2(\Omega))$ and $a.e.$ in $\Omega_{T^*}$. Passing to the limit in \eqref{pn} and in estimates \eqref{13:35}-\eqref{13:35bis}, we deduce that $u$ is a solution of our problem with the desired properties.\\
\textbf{Step 2}  Assume now $\mu\in[1,2)$. Thanks to assumption \eqref{thetabig}, Step 1 provides us with a sequence $v_n\in W^2_{loc}([0,T^*))$ that solve $v_n(0,x)=T_n(u_0)$ and
\[
\int_0^\tau\langle \partial_t v_n ,\varphi \rangle dt
+
 \int_{\Omega_\tau} M(t,x) \nabla v_n \cdot \nabla \varphi \dx x\dx t
=  \int_{\Omega_\tau} |v_n|^{\theta}v_n E \cdot  \nabla \varphi \dx x\dx t 
\]
for any $\tau\in(0,T^*)$ and for any $\varphi\in L^{2}_{loc}(0,T^*, W^{1,2}_0(\Omega))\cap L^{\infty}(\Omega_{T^*})$, where $T^*$ is the one given by Proposition \ref{12:15prop}. Applying the comparison principle provided by Theorem \ref{comparison} to $v_n$ and $v_m$, we get that for all $t\in(0,T^*)$
\[
\int_{\Omega}|v_n(t,x)-v_m(t,x)|dx\le \int_{\Omega}|T_n(u_0)(x)-T_m(u_0)(x)|dx,
\]
that implies that the sequence $v_n$ is Cauchy in $C([0,T^*), L^1(\Omega))$. Now we claim that we can apply Lemma \ref{19:03} with the choice $v=v_n$. Clearly we have $v_n\in W^{2}_{loc}(0,T)$. Moreover, thanks again to Step 1 and the fact that $1/r+ \theta /2<1/r + \theta /\mu<1/N$, we know that for any $q>2$
\[
\|v_n(t)\|_{L^q (\Omega )}\leq  C(t)t^{-\frac N2 (\frac{1}{2}-\frac 1q)} \qquad \mbox{ for } t\in (0,T^*),
\]
where $C(t)\to\infty$ as $t\to T^*$
However since $\Omega$ is bounded this implies that
\[
v_n\in L^{\infty}_{loc}((0,T^*), \elle q) \quad \forall q\ge1.
\]
From this property we end the proof of the claim. Now, following the very same procedure of Proposition \ref{12:15prop} (we omit further details for the sake of brevity), we deduce that 
\be\label{ciaone}
\|u(t)\|_{\elle{\mu}}\le \|u_0\|_{\elle{\mu}}\left(\frac{T^*}{T^*-t }\right)^{\gamma}\qquad  \|v_n(t)\|_{L^m (\Omega )} \le  C(t)t^{-\frac N2 \big(\frac1\mu-\frac 1m\big)},
 \ee
for $t\in(0,T^*)$. Taking $v_n \chi_{(\epsilon,\tau)}$, with $\epsilon>0$ and $\tau<T^*$, as a test function, it follows that
\[
\io v_n^2(\tau)dx+\frac{\alpha}{2}\int_{\epsilon}^{\tau}\io|\nabla v_n|^2 dxdt \le \int_{\epsilon}^{\tau}\io |v_n|^{2(\theta+1)}|E|^2|dxdt+\io v_n(\epsilon)^2dx.
\]
Estimate \eqref{ciaone} with $m=2^*>2>\mu$ reads as
\[
\|v_n(t)\|_{L^{2^*} (\Omega )}\leq C(t)  \|u_0\|_{L^{\mu}(\Omega)}  \epsilon^{\frac N2 (\frac 1{2^*}-\frac{1}{\mu})}
\]
Therefore, we get that
\[
\begin{split}
\int_\epsilon^{\tau}\int_{\Omega}|v_n(t)|^{2(\theta+1)}|E|^2dxdt&\le\int_0^{\tau} \|\un(t)\|_{\elle{2^*}}^2 \|v_n(t)\|_{\elle1}^{2(\theta+1)-2}\|E(t)\|_{\elle r}^2dt\\
&\le C(\epsilon,\tau) \|u_0\|_{L^{\mu}(\Omega)}^2\|u_0\|_{\elle1}^{2(\theta+1)-2}\|E\|_{L^{\infty}((0,T^*),\elle r)}^2,
\end{split}
\]
with the constant $C(\epsilon, T)$ that diverges both as $\epsilon\to 0$ and $\tau\to T^*$. This means that $v_n$ is bounded in $L^2_{loc}((0,T^*), W^{1,2}_0(\Omega))$ and, in turn, that $\partial_t v_n$ is bounded in $L^2_{loc}((0,T^*), W^{-1,2}(\Omega))$. All this information implies that there exists $u\in W^2_{loc}((0,T^*))\cap C([0,T^*),\elle 1))$ limit of the sequence $v_n$ in the associate topology. Passing to the limit as $p\to \infty$ in the weak formulation solved by $v_n$ and in estimates \eqref{ciaone} implies that $u$ is a distributional solution of our problem. 

In both Steps, uniqueness follows by the comparison principle (Theorem \ref{comparison}), the fact that $u\in  W^2_{loc}(0,T^*)$, $|E||u|^{\theta}\in L^2_{loc}((0,T^*);\elle2)$, and Remark \ref{helpingremark}.
\end{proof}

\section{Proof of Theorem \ref{exist2}}
In this section we obtain the existence of solutions to Problem  \eqref{problemintro} in cases where the exponent $\theta$ is not `too far' from 
$0$, proving Theorem \ref{exist2}.

\begin{lemma}\label{Lemmauno} Keep the same assumptions of Theorem \ref{exist2} and let $\{\un\}$ be the sequence of solution to the approximating problems $(\mathcal P_n)$. Therefore, for any $q\in(1,\frac{N+2}{2})$, we have that
\[
\|\un\|_{L^{\infty}(0,T,L^{q^{\star\star}\frac{N}{N+2}}(\Omega))}+ \|\un\|_{L^{q^{\star\star}}(\Omega_T)}\le C_1 \left(\|f\|_{L^q(\Omega_{T})}+\|u_0\|_{L^{q^{\star\star}\frac{N}{N+2}}(\Omega)}\right).
\]
Moreover, for any $q\in(1,\sigma']$ with $\sigma=2\frac{N+2} N$, we have that
\[
 \|\nabla \un\|_{L^{q^{\star}}(\Omega_T)} \le  C_2 \left(\|f\|_{L^q(\Omega_{T})}+\|u_0\|_{L^{q^{\star\star}\frac{N}{N+2}}(\Omega)}\right).
\]
Here the positive constant $C_1,C_2$ depend on $N, T, \alpha, \|f\|_{L^1(\Omega_T)}, \|u_0\|_{L^1(\Omega)}, \|E\|_{L^r(\Omega)}$.
\end{lemma}

\begin{proof} Take $q\in(1,\frac{N+2}{N})$ and consider the function $ \psi_{\varepsilon}(v)$ defined as
\[
\psi_\varepsilon(v):=\frac{1}{2\gamma+1}\left[ (T_k(|v|)+\varepsilon)^{2\gamma+1} - \varepsilon^{2\gamma+1} \right]{\rm sign}(v),
\]
where  $\varepsilon >0$ and $\gamma$ given by,
\begin{equation}\label{gammadef}
\gamma=\frac{qN-2N-4+4q}{(N+2-2q)2}.
\end{equation}
In the considered range of $q$, the main properties of $\gamma$ are
\be\label{gammaproperties}
2\gamma+1>0, \quad 2\gamma+2=q^{\star\star}\frac{N}{N+2}, \quad q'(2\gamma+1)=q^{\star\star},
\ee
and  $\gamma>2$ if and only if $ q>2\frac{N+2}{N+4}$. We also define, for any $u\in\mathbb R$,
$$\Psi_\varepsilon(u):=\int_0^u \psi_\varepsilon(v)dv.$$
Take $\tau \in (0,h)$, with $h\in(0,T)$ to be chosen in the sequel, and use $\psi_\varepsilon (u_n) \chi_{(0,\tau)}$ as a test function in \eqref{pn}.
We obtain
\begin{equation}\label{est1}
\begin{split}
&\int_\Omega 	 \Psi_\varepsilon(u_n(\tau))dx + \alpha \int_{\Omega_\tau} |\nabla T_k(u_n)|^2  (|u_n|+\varepsilon)^{2\gamma} dx dt
\\
 \le&
\int_{\Omega_\tau} |E||u_n|^{\theta+1} |\nabla T_k(u_n)|  (|u_n|+\varepsilon)^{2\gamma} dx dt
+\int_{\Omega_\tau} |f||\psi_\varepsilon(u_n)|dx dt + \int _ \Omega \Psi_\varepsilon(u_0)dx\\
 \le& \frac{1}{2\alpha} \int_{\Omega_\tau} |E|^2(T_k(|\un|)+\epsilon)^{2(\theta+1)+2\gamma}dx dt+ \frac{\alpha}2 \int_{\Omega_\tau} |\nabla T_k(u_n)|^2  (|u_n|+\varepsilon)^{2\gamma} dx dt\\
&\quad +\int_{\Omega_\tau} |f||\psi_\varepsilon(u_n)|dx dt + \int _ \Omega \Psi_\varepsilon(u_0)dx,
\end{split}
\end{equation}
where we applied Young inequality in the drift term.
Summing up the similar terms and taking the supremum with respect to $\tau\in(0,h)$, \eqref{est1} becomes
\begin{equation}\label{est2}
\begin{split}
\sup_{\tau\in(0,{ h })}\int_\Omega 	& \Psi_\varepsilon(u_n(\tau))dx +  \int_{\Omega_{ h }} |\nabla T_k(u_n)|^2  (T_k(|u_n|)+\varepsilon)^{2\gamma} dx dt\\
& \le C_{\alpha,q}\left[ \int_{\Omega_{ h }} |E|^2(|\un|+\epsilon)^{2(\theta+1)+2\gamma} +\int_{\Omega_{ h }} |f||\psi_\varepsilon(u_n)| + \int _ \Omega \Psi_\varepsilon(u_0)dx \right].
\end{split}
\end{equation}
Notice now that the left hand side above is bounded with respect to $\varepsilon$ (at this stage both $n$ and $k$ are fixed). 
Namely,
\[
\frac{1}{(\gamma+1)^2}\int_{\Omega_{ h }} |\nabla [(T_k(|u_n|)+\varepsilon)^{1+\gamma}-\varepsilon^{1+\gamma}]|^2  dx dt =\int_{\Omega_{ h }} |\nabla T_k(u_n)|^2  (|u_n|+\varepsilon)^{2\gamma} dx dt\le C.
\]
Therefore, the sequence $\phi_{\epsilon}=(|u_n|+\varepsilon)^{1+\gamma}-\varepsilon^{1+\gamma}$ is bounded in $L^2(0,h;W^{1,2}_0(\Omega))$ and it converges $a.e$ in $\Omega_T$ to $|\un|^{\gamma+1}$. Therefore, using Fatou's Lemma on the left and the dominate convergence Theorem on the right, we can take $\epsilon\to 0$ in \eqref{est2}, to get
\begin{equation}\label{est2bis}
\begin{split}
\sup_{\tau\in(0,h)}\int_{\Omega}&T_k(|\un({\tau})|)^{2(\gamma+1)}+\int_{\Omega_{ h }} |\nabla (T_k(u_n)^{\gamma+1})|^2\le\\ 
\le& C_{\alpha,q}\left[\int_{\Omega_{ h }} |E|^2T_k(|u_n|)^{2(\theta+1)+2\gamma} +\int_{\Omega_{ h }} |f|T_k(|u_n|)^{2\gamma+1} + \int _ \Omega |u_0|^{2\gamma+2}dx\right].
\end{split}
\end{equation}
Applying Gagliardo-Nirenberg inequality to the function $T_k(|u_n|)^{\gamma+1}$, it follows that
\begin{equation}\label{est3}
\begin{split}
 \|T_k(\un)\|_{L^{q^{\star\star}} (\Omega_{ h }) }^{2\gamma+2}& = \left(\int_{\Omega_{ h }}   T_k(|u_n|)^{\sigma(\gamma+1)}\right)^{\frac{N}{N+2}}\\
& \le C_N \left( \sup_{\tau\in(0, h )}\int_{\Omega}T_k(|\un({ h })|)^{2(\gamma+1)}\right)^{\frac{2}{N+2}} \left(\int_{\Omega_{ h }} |\nabla (T_k(|u_n|)^{\gamma+1})|^2\right)^{\frac{N}{N+2}} \\ 
&\le A \left[ \int_{\Omega_{ h }} |E|^2T_k(|u_n|)^{2(\theta+1)+2\gamma} +\int_{\Omega_{ h }} |f|T_k(|u_n|)^{2\gamma+1} + \int _ \Omega |u_0|^{2\gamma+2}dx\right],
\end{split}
\end{equation}
for some positive constant $A=A(N,\alpha,q)$.
In order to estimate the contribution of the drift term in the right hand side above, let us recall \eqref{8}. Therefore, it follows that,
\begin{equation*}
\begin{split}
\int_{\Omega_{ h }} |E|^2T_k(|u_n|)^{2(\theta+1)+2\gamma}&\le 
\|E\|^2_{L^r(\Omega_{ h })} 
\left\|u_n\right\|^{2 \theta }_{L^1(\Omega_{ h })}
\left\|  T_k(u_n)\right\|^{2\gamma+2}_{L^{q^{\ast\ast}}(\Omega_{ h })}\\
&\le  h ^{2\theta} \|E\|^2_{L^r(\Omega_T)} M_0^{2 \theta } \left\|  T_k(u_n)
\right\|^{2\gamma+2}_{L^{q^{\ast\ast}}(\Omega_h)},
\end{split}
\end{equation*}
where we have also use H\"older inequality and \eqref{gammaproperties}. Moreover, we have that
\[
\begin{split}
A \int_{\Omega_{ h }} |f|T_k(|u_n|)^{2\gamma+1} & \le \tilde A \|f\|_{L^q(\Omega_T)}
\|T_k(\un)\|_{L^{q^{\star\star}} (\Omega_h) }^{2\gamma+1}\\ 
& \le \tilde A \|f\|_{L^q(\Omega_T)}^{2\gamma+2}+ \frac14 \|T_k(\un)\|_{L^{q^{\star\star}} (\Omega_h) }^{2\gamma+2}.
\end{split}
\]
Plugging this two pieces of information in \eqref{est3}, we obtain
\[
\begin{split}
\|T_k(\un)\|_{L^{q^{\star\star}} (\Omega_{ h }) }^{2\gamma+2}&\le A  h ^{2\theta} \|E\|^2_{L^r(\Omega_T)} M_0^{2 \theta } \left\|  T_k(\un) \right\|^{2\gamma+2}_{L^{q^{\ast\ast}}(\Omega_h)}\\
&+\tilde A \|f\|_{L^q(\Omega_T)}^{2\gamma+2}+ \frac14 \|T_k(\un)\|_{L^{q^{\star\star}} (\Omega_h) }^{2\gamma+2}\\
&+A\int _ \Omega |u_0|^{2\gamma+2}dx.
\end{split}
\]
Choosing $ h $ such that
\begin{equation}\label{sceltah}
 A  h ^{2\theta} \|E\|^2_{L^r(\Omega_T)} M_0^{2\theta}\le \frac14,
\end{equation}
we obtain the following estimate (letting $k\to\infty$)
\[
\|\un\|_{L^{q^{\star\star}} (\Omega_{ h }) }^{2\gamma+2}\le B \left( \|f\|_{L^q(\Omega_T)}^{2\gamma+2}+\int _ \Omega |u_0|^{2\gamma+2}dx\right).
\]

Going back to \eqref{est2bis} and letting again $k\to\infty$, we get
\be \label{dequation}
\sup_{\tau\in (0, h )}\int_{\Omega}|\un({\tau})|^{2(\gamma+1)}+\int_{\Omega_{ h }} |\nabla (|u_n|^{\gamma+1})|^2\le C\left( \|f\|_{L^q(\Omega_T)}^{2\gamma+2}+\int _ \Omega |u_0|^{2\gamma+2}dx \right).
\ee
Clearly, if \eqref{sceltah} is satisfied with $h\equiv T$, the proof of this first part of the Lemma is complete.\\
 If not, we proceed dividing the time interval (0,T) in suitably small slices.  For it, let us take 
\[
h= \frac{1}{ ( 4 A   \|E\|^2_{L^r(\Omega_T)} M_0^{2 \theta })^{1/(2(\theta+1)-2)}},
\]
and set $t_0=0$ and $t_{i}=t_{i-1}+h$ for all $i=1,2,\cdots, k$ where $k$ is the integer part of $T/h$; let us also set $t_{k+1}=T$. We claim that, for any $i\in \mathbb N$ with $i\le k+1$,
\begin{equation}\label{induzione}
\sup_{\tau\in ( t_{i-1} ,t_{i})}\int_{\Omega}|\un({\tau})|^{2(\gamma+1)}+\int_{(t_{i-1},t_i)\times\Omega} |\nabla (|u_n|^{\gamma+1})|^2\le \sum_{j=1}^iD^j \|f\|_{L^q(\Omega_{T})}^{2\gamma+2}+ D^i\int _ \Omega |u_0|^{2\gamma+2}dx,
\end{equation}
where $D$ is the constant in \eqref{dequation}. We have already shown that the claim is true for $i=1$. Therefore assume that \eqref{induzione} holds for some $i=1,\cdots, k$ and let us prove the same inequality for $i+1$. Taking $\psi_{\varepsilon}(\un)\chi_{(t_i,t_{i+1})}$ as a test function in \eqref{pn}, and following the same procedure of the first step of the proof, we obtain that
\[
\sup_{\tau\in (t_i, t_{i+1} )}\int_{\Omega}|\un({\tau})|^{2(\gamma+1)}+\int_{(t_i,t_{i+t})\times\Omega} |\nabla (|u_n|^{\gamma+1})|^2\le D\left( \|f\|_{L^q(\Omega_T)}^{2\gamma+2}+\int _ \Omega |\un(t_{i})|^{2\gamma+2}dx \right).
\]
Let us stress that the last integral makes sense since $\un\in C(0,T,L^{2}(\Omega))$. For the same reason we can deduce from \eqref{induzione} that
\[
\int _ \Omega |\un(t_{i})|^{2\gamma+2}dx \le \sum_{j=1}^iD^j \|f\|_{L^q(\Omega_{T})}^{2\gamma+2}+ D^i \int _ \Omega |u_0|^{2\gamma+2}dx.
\]
Therefore, it follows that
\[
\begin{split}
\sup_{\tau\in (t_i, t_{i+1} )}\int_{\Omega}|\un({\tau})|^{2(\gamma+1)}&+\int_{(t_i,t_{i+t})\times\Omega} |\nabla (|u_n|^{\gamma+1})|^2\\
&\le D\left( \|f\|_{L^q(\Omega_T)}^{2\gamma+2}+\sum_{j=1}^iD^j \|f\|_{L^q(\Omega_{T})}^{2\gamma+2}+ D^i \int _ \Omega |u_0|^{2\gamma+2}dx\right)\\
&\le  \sum_{j=1}^{i+1}D^j \|f\|_{L^q(\Omega_{T})}^{2\gamma+2}+ D^{i+1} \int _ \Omega |u_0|^{2\gamma+2}dx,
\end{split}
\]
and the claim follows by finite induction.\\
Summing up estimate \eqref{induzione} for each time interval $(t_i,t_{i+1})$, we deduce that 
\begin{equation}\label{induzione2}
\sup_{\tau\in ( 0 ,T)}\int_{\Omega}|\un({\tau})|^{2(\gamma+1)}+\int_{\Omega_T} |\nabla (|u_n|^{\gamma+1})|^2\le C( \|f\|_{L^q(\Omega_{T})}^{2\gamma+2}+ \int _ \Omega |u_0|^{2\gamma+2}dx),
\end{equation}
and, applying once more time Gagliardo-Nirenberg inequality, 
\begin{equation}\label{stimaustar}
\begin{split}
 \|\un\|_{L^{q^{\star\star}}(\Omega_T)}\le C \left(\|f\|_{L^q(\Omega_{T})}+\|u_0\|_{L^{q^{\star\star}\frac{N}{N+2}}(\Omega)}\right),
\end{split}
\end{equation}
for a suitable constant $C$.
Having at hand the two inequalities above, it is standard to deduce the estimate for the gradients (see for instance \cite{BDGO2}). Here we give just a quick sketch for the convenience of the reader:
\begin{equation*}
\begin{split}
\left(\int_{\Omega_T} |
\nabla u_n|^{q^{\star}}dxdt\right)^{\frac1{q^{\star}}} & \leq c \left(\int_{\Omega_T }
 |\nabla (|u_n|^{\gamma+1})|^2\right)^{\frac {1}2} \left( \int_{\Omega_T  }  |u_n| ^{q^{\star\star}}\right) ^{\frac1{q^{\star}}-\frac12}\\
& \le C\left(  \|f\|_{L^q(\Omega_{T})} +\|u_0\|_{L^{q^{\star\star}\frac{N}{N+2}}(\Omega)}   \right)^{\gamma+1+q^{\star\star}\left(\frac1{q^{\star}}-\frac12\right)}
\\
&=C\left(  \|f\|_{L^q(\Omega_{T})} +\|u_0\|_{L^{q^{\star\star}\frac{N}{N+2}}(\Omega)}   \right).
\end{split}
\end{equation*}

\end{proof}

\begin{lemma} \label{stimamarcin}
Let assumptions of Lemma \ref{Lemmauno} be in force with $q=1$. Then,  every solution $u_n \in  \LBoch \infty 0 T {2} \Omega\cap\Boch {2} 0 T 1 2 \Omega$ to problem \eqref{pn} satisfies the estimate
\[
\|u_n\|_{\mathcal M^{1^{\star\star}} (\Omega_T)}+\| \nabla u_n \|_{\mathcal M^{1^{\star}} (\Omega_T)} \le C\big( \|f\|_{L^1(\Omega_{T})},\|u_0\|_{L^{1^{\star\star}}(\Omega)}\big).
\]
\end{lemma}

\begin{proof}
Given $k\ge0$, direct computation shows that
\[
\frac 1 {2} 
\left|T_k(z)\right|^{2}
\le
\Phi_k(z)
\le  k|z|, \quad \mbox{where} \quad \Phi_k(z):=\int_0^z 
T_k(\zeta) \dx \zeta.
\]
Therefore, taking $\tau \in (0,h)$ with $h\in(0,T)$ to be chosen in the sequel, and using  $\varphi=T_k(u_n)\chi_{0,\tau}$ as a test function in \eqref{pn}, it follows that
\begin{equation}\label{2.1}
\begin{split}
 \int_\Omega |T_k(u_n(x,\tau))|^2\dx x&
+
\int_{\Omega_{\tau}} |\nabla T_k(u_n)|\dx x\dx t
\\
&\le C_{\alpha}\left[\int_{\Omega_{\tau}}     |E||u_n|^{\theta+1} |\nabla T_k(u_n)|     \dx x\dx t+
k\int_\Omega |u_0(x)|\dx x
+
k\int_{\Omega_{\tau}} |f| \dx x\dx t\right].
\end{split}
\end{equation}
Recalling \eqref{8}, we estimate the first integral in the right hand side above as follows
\begin{equation*}
\begin{split}
\int_{\Omega_{\tau}} |E| |T_k(u_n)|^{\theta+1}|\nabla T_k(u_n)|   \dx x \dx \tau &\le  \|E\|_{L^r(\Om_{h})}\|u_n\|^{\theta}_{L^1(\Om_{h})}\|T_k(u_n)\|_{L^\sigma(\Om_{\tau})}\|\nabla T_k(u_n)\|_{L^2(\Om_{\tau})}\\
& \le (hM_0)^{\theta} \|E\|_{L^r(\Om_{h})}\|T_k(u_n)\|_{L^\sigma(\Om_{\tau})}\|\nabla T_k(u_n)\|_{L^2(\Om_{\tau})}.
\end{split}
\end{equation*}
Moreover, using the Gagliardo--Nirenberg inequality, we deduce
\begin{equation}\label{31bis}
\begin{split}
\int_{\Omega_t}  &  |E| |T_k(u_n)|^{\theta+1}|\nabla T_k(u_n)|      \dx x \dx \tau \\
&\le  C_1  (hM_0)^{\theta}    \|E\|_{L^r(\Om_{h})} 
(
\| T_k(u_n)\|_{L^\infty(0,T,L^2(\Om_{\tau}))}+
\|\nabla T_k(u_n)\|_{L^2(\Om_{\tau})}
)
\|\nabla T_k(u_n)\|_{L^2(\Om_{\tau})}     \\
&\le  C_2  (hM_0)^{\theta}    \|E\|_{L^r(\Om_{h})} 
(
\| T_k(u_n)\|^2_{L^\infty(0,T,L^2(\Om_{\tau}))}+\|\nabla T_k(u_n)\|^2_{L^2(\Om_{\tau})}).
\end{split}
\end{equation}
Let us take $h$ such that
\begin{equation}\label{smallcon}
 C_{\alpha}C_2  (hM_0)^{\theta}    \|E\|_{L^r(\Om_{T})} < \frac{\alpha}{2},
\end{equation}
and plug \eqref{31bis} into \eqref{2.1}. We conclude that
\begin{equation}\label{2.2bis}
\begin{split}
\sup_{0<\tau\le h}
 \int_{\Omega }       \left|T_k\left(u_n(x,\tau)\right)\right|^{ 2 } \dx x
+
\int_{\Omega_{h}} |\nabla T_k(u_n)|^2 \dx x\dx \tau
\le C_4 \left(
\|u_0\|_{L^1(\Om)} + \|f\|_{L^1(\Om_T)} 
\right)k.
\end{split}
\end{equation}
If condition \eqref{smallcon} is valid with $h\equiv T$ we are done. Otherwise we shall slide the interval $(0,T)$ as follows: take
\[
h=\left(\frac{\alpha}{2C_{\alpha}C_2\|E\|_{L^r(\Om_{T})}}\right)^{\frac{1}{\theta}}\frac{1}{M_0},
\]
and set $t_0=0$ and $t_{i}=t_{i-1}+h$ for all $i=1,2,\cdots, k$ where $k$ is the integer part of $T/h$; let us also set $t_{k+1}=T$. Testing \eqref{pn} with $T_k(\un)\chi_{t_i,t_{i+1}}$, we get 
\begin{equation*}
\begin{split}
\sup_{t_i<\tau\le t_i+1 }
 \int_{\Omega }       \left|T_k\left(u_n(x,\tau)\right)\right|^{ 2 } \dx x
+
\int_{t_i}^{t_{i+1}}\int_{\Omega} |\nabla T_k(u_n)|^2 \dx x\dx \tau
\le C_5 \left(
\|u_0\|_{L^1(\Om)} + \|f\|_{L^1(\Om_T)} 
\right)k.
\end{split}
\end{equation*}

Reasoning as in the proof of Lemma \ref{Lemmauno} (from \eqref{induzione} onward), we deduce that 
\begin{equation*}\label{2.2boot}
\begin{split}
\sup_{0<t<T}
 \int_{\Omega }       \left|T_k\left(u_n(t,x)\right)\right|^{ 2 } \dx x
+
\int_{\Omega_{T}} |\nabla T_k(u_n)|^2 \dx x\dx \tau
\le C_6 \left(
\|u_0\|_{L^1(\Om)} + \|f\|_{L^1(\Om_T)} 
\right)k.
\end{split}
\end{equation*}
It is well known that such an inequality the desired result (see for instance \cite{BG}). Indeed, using again Gagliardo-Nirenberg inequality, we obtain
\[
k^{2+\frac4N}|\{(t,x)\in\Omega_T \ : \ |\un|\ge k\}|\le \int_{\Omega_T}|T_k(\un)|^{2\frac{N+2}{N}}\le C k^{1+\frac2N}M_0^{1+\frac2N}.
\] 
Concerning the gradient estimate it follows that
\[
\begin{split}
|\{|\nabla \un|>\lambda\}|& \le |\{|\nabla \un|>\lambda, |\un|\ge k\}|+|\{|\nabla \un|>\lambda, |\un|\le k\}|\\
 &\le C\left(\frac{1}{k^{1+\frac2N}}+\frac{k}{\lambda^2}\right)= \frac{C}{\lambda^{\frac{N+2}{N+1}}},
\end{split}
\]
where the last equality follows from the choice $k=\lambda^{\frac{N}{N+1}}$.
\end{proof}

We are now in position to give the proof of Theorem \ref{exist2}. 
\begin{proof}[Proof of Theorem \ref{exist2}]

Let us focus on the case $(iii)$ of the Theorem, namely $q\in[\sigma',\frac{N+2}{2})$. Lemma \ref{Lemmauno} implies that there exists $u$ such that, up to a not relabelled subsequence,
\[
u_n\rightharpoonup u \qquad \mbox{ in }L^{2}(0,T,W_0^{1,2} (\Omega))\cap L^{q^{\star\star}}(\Omega_T).
\]
Moreover since
\[
\frac{1}{r}+\theta+\frac{1}{q^{\star\star}}\le \frac12
\]
we also have that the sequence $E_ng_n(\un)$ si bounded in $(L^2(\Omega_T))^{N}$. These two pieces of information imply that the sequence $\partial_t \un$ is bounded in $L^2(0,T, W^{-1,2}(\Omega))$, that in turn implies that $\un\to u$ in $L^1(\Omega_T)$ and $C([0,T);L^2(\Omega))$. To conclude that $u$ is a distributional solution, notice that, for every test function $\varphi\in  C^{\infty}([0,T]\times\Omega)$ and $0<t_1<t_2<T$ we have:
\begin{eqnarray*} 
\io \big(u(t_2)\varphi(t_2)-u(t_1)\varphi(t_1)\big)\dx x-
 \int_{t_1}^{t_2}\int_{\Omega} u\partial_t\varphi\\
+\int_{t_1}^{t_2}\int_{\Omega} \big(M\nabla u+E|u|^{\theta}u \big)\nabla\varphi\dx x \dx t=\int_{t_1}^{t_2}\int_{\Omega}  f (t,x) \varphi  \dx x \dx t.
\end{eqnarray*} 
Taking the limit as $n\to\infty$ we obtain the existence of a solution. Uniqueness follows by the comparison principle of Theorem \ref{comparison} and the regularity of $u$ (see Remark \ref{helpingremark}).\\

Let us consider now case $(i)$, that is $f\in L^1(\Omega_T)$ and $u_0\in L^1(\Omega)$. Notice that, thanks to the previous step, we can now solve problem \eqref{pn} with the choice $g_n(s)=|s|^{\theta}s$, namely truncating just $f$ and $u_0$. Therefore, naming, with a slight abuse of notation, the solutions of the corresponding problems again as $\un$, we can now use the comparison principle with $u_n$ and $u_m$ to get
\[
\int_{\Omega}|u_n(\tau)-u_m(\tau)|\le \|f_n-f_m\|_{L^1(\Omega_T)}+\|u_{0,n}-u_{0,m}\|_{\elle1},
\]
that implies that the sequence $u_n$ is Cauchy in $C([0,T], \elle1)$.
Moreover, using the estimate given by Lemma \ref{stimamarcin} for this sequence of $\un$, we deduce that there exists $u\in L^{m}(0,T,W_0^{1,m} (\Omega))$, for all $m<\frac{N+2}{N+1}$, such that, up to a (not relabeld) subsequence,
\begin{equation}\label{weakconv}
u_n\rightharpoonup u \qquad \mbox{ in }L^{m}(0,T,W_0^{1,m} (\Omega)) \quad \forall \ m<\frac{N+2}{N+1}.
\end{equation}
Moreover, also the sequence $|E_n||g_n(\un)|$ is bounded in $L^{m}(\Omega_T)$, again for all $m<\frac{N+2}{N+1}$. Indeed, since
assumption \eqref{closelin} implies that
\[
\frac1r+\theta+\frac{N}{N+2}=\frac{1}{N+2}+\frac{N}{N+2}=\frac{N+1}{N+2}>1,
\]
for any $m\in[1,\frac{N+2}{N+1})$ there exists $s\in [1,\frac{N+2}{N})$ such that
\[
\frac1r+\theta+\frac{1}{s}=\frac{1}{m}<1.
\]
Using H\"older inequality, we therefore obtain that
\[
\left(\int_{\Omega}\big(|E_n||\un|^{\theta+1}\big)^{m}\right)^{\frac1m}\le \|E\|_{L^{r}(\Omega_T)}\|\un\|_{L^1(\Omega_T)}\|\un\|_{L^{s}(\Omega_T)}.
\]

These two pieces of information imply that the sequence $\partial_t \un$ is bounded in $L^1(0,T, W^{-1,m}(\Omega))+ L^1(\Omega_T)$, with $m$ as above, and thanks to the compactness results of Aubin's type for non reflexive spaces (see Corollary 4 in \cite{Sim}), we deduce that 
\begin{equation*}
u_n\rightarrow u \qquad  a.e. \mbox{ in } \Omega_T.
\end{equation*}
This $a.e.$ convergence and the boundedness showed above imply that
\begin{equation}\label{driftcon}
E_ng_n(\un) \to E|u|^{\theta}u \quad \mbox{ in } (L^{m}(\Omega_T))^N
\end{equation}
for any $m<\frac{N+2}{N+1}$. \\
To show that $u$ is a distributional solution, let us go back to the equation solved by the $\un$. It follows that, for every test function $\varphi\in  C^\infty(\bar \Omega_T)$ such that ${\rm supp }\,\varphi \subset [0,T)\times \Omega$ we have:
\[
- \int_{\Omega_T} \un \varphi_t 
+
 \int_{\Omega_T} M(x) \nabla \un \cdot \nabla \varphi 
=  \int_{\Omega_T}  E_n
g_n(\un)  \nabla \varphi    +\int_\Omega u_0 \varphi (x,0)
+
 \int_{\Omega_T}
 f_n  \varphi 
\] 
The weak convergence in \eqref{weakconv} is enough to pass to the limit in the left hand side, while, to deal with the drift term we use \eqref{driftcon}. Taking the liminf as $n\to\infty$ in the estimates of Lemma \ref{stimamarcin}, we conclude this part of the proof.\\
The case $(ii)$ follows exactly in the same way with the only difference of using Lemma \ref{Lemmauno} instead of Lemma \ref{stimamarcin}. 
\end{proof}

\section{Proof of Theorem \ref{thmSte}}

To prove Theorem \ref{thmSte} we shall use a fixed point argument. To this aim, for a fixed $v\in L^{q^{\star\star}}(\Om_T)$ we introduce the following problem 

\begin{equation} \label{problem}
\begin{cases}
\partial_t u-\div(M(t,x)\nabla u)= -\div(E |v|^{\theta} v )+f(x) & \quad  \mbox{in } \Om_T,\\
u = 0 & \quad \mbox{on }  (0,T) \times \partial \Omega ,\\
u(x,0)=u_0(x) & \quad \mbox{in } \Omega.
\end{cases}
\end{equation}

We observe that previous problem admits an unique solution $u= \mathscr F (v) \in L^{q^{\star\star}}(\Om_T)\cap L^2(0,T,W^{1,2}_0(\Omega))$. Indeed, thanks to the linearity we can rewrite $u=u_1+u_2$ where $u_1$ solves the problem 
\begin{equation*} 
\begin{cases}
\partial_t u_1-\div(M(t,x)\nabla u_1)= -\div(E |v|^{\theta} v ) & \quad  \mbox{in } \Om_T,\\
u_1 = 0 & \quad \mbox{on }  (0,T) \times \partial \Omega,\\
u_1(x,0)=0 & \quad \mbox{in } \Omega.
\end{cases}
\end{equation*}   
In view of \eqref{farfromlin} we have $F=E |v|^{\theta} v \in L^{q^\star}(\Omega_T) $, than we can apply  \cite{BDGO2}, Remark 2.6; moreover such solution $u_1$ is unique since in our assumption $q^{\star}\geq 2.$  
On the other hand $u_2$ solves the problem
\begin{equation*} 
\begin{cases}
\partial_t u_2-\div(M(t,x)\nabla u_2)= f(x) & \quad  \mbox{in } \Om_T,\\
u_2 = 0 & \quad \mbox{on }  (0,T) \times \partial \Omega,\\
u_2(x,0)=u_0(x) & \quad \mbox{in } \Omega.
\end{cases}
\end{equation*}
Here we can use previous Theorem \ref{exist2} $(iii)$ with $E=0.$

So, we consider the map 
\begin{equation}\label{mappa}
\mathscr F:L^{q^{\star\star}}(\Om_T) \to  L^{q^{\star\star}}(\Om_T)
\end{equation}
 that associates to  any $v\in L^{q^{\star\star}}(\Om_T)$  the unique solution $u= \mathscr F (v) \in L^{q^{\star\star}}(\Om_T)\cap L^2(0,T,W^{1,2}_0(\Omega))$ of Problem \ref{problem}.

\begin{lemma}\label{stimafrozen}
The unique solution to \eqref{problem} satisfies
\be\label{stimalk}
\|G_l(u)\|_{L^{q^{\star\star}}(\Om_T)}\le  C_{\alpha,q} \left[\mathcal{C}_{E,l}\|v\|_{L^{q^{\star\star}}(\Om_T)}^{\theta+1}+ \mathcal{C}_{f,l}+ \|u_0\|_{L^{q^{\star\star}\frac{N}{N+2}}(\Omega)} \right]\quad \forall \ 0\le l,
\ee
where  $\mathcal{C}_{f,l}=(\int_{\Om_T\cap\{|u|\ge l\}}|f|^q)^\frac1q$ and $\mathcal{C}_{E,l}=(\int_{\Om_T\cap\{|u|>l\}}|E|^r)^\frac1r$.
\end{lemma}

\begin{proof}{Proof of Lemma \ref{stimafrozen}}
Let us set $u_{k,l}=T_k(G_l(u))$ and take $\frac{1}{2\gamma+1}|u_{k,l}|^{2\gamma}u_{k,l} \, \chi_{(0,\tau)}$, where $\gamma$ is defined in \eqref{gammadef}, 
as a test function in the weak formulation of \eqref{problem}. We stress that in this case $\gamma>0$. We get
\be\label{30-10}
\begin{split}
\frac{1}{2(\gamma+1)(2\gamma+1)}\sup_{\tau\in(0,T)}\io |u_{k,l}(\tau)|^{2\gamma+2}+  \alpha\iqT |\nabla u_{k,l}|^{2}|u_{k,l}|^{2\gamma}\\ \le \iqT |E||v|^{\theta+1}|\nabla u_{k,l}||u_{k,l}|^{2\gamma}+\frac{1}{2\gamma+1}\iqT|f||u_{k,l}|^{2\gamma+1}+\io |u_0|^{2\gamma+2}.
\end{split}
\ee
Let us estimate the first two integrals on the right hand side above: At first notice that
\be\label{cip}
\begin{split}
\iqT |E||v|^{\theta+1}|\nabla u_{k,l}||u_{k,l}|^{\gamma}|u_{k,l}|^{\gamma} & \le  \mathcal{C}_{E,l}\|v\|_{L^{q^{\star\star}}(\Om_T)}^{\theta+1}\||\nabla u_{k,l}||u_{k,l}|^{\gamma}\|_{L^{2}(\Om_T)}\|u_{k,l}\|_{L^{q^{\star\star}}(\Om_T)}^{\gamma}\\
& \le  \frac{\mathcal{C}_{E,l}^2}{2\alpha}\|v\|_{L^{q^{\star\star}}(\Om_T)}^{2(\theta+1)}\|u_{k,l}\|_{L^{q^{\star\star}}(\Om_T)}^{2\gamma}+\frac{\alpha}{2}\||\nabla u_{k,l}||w|^{\gamma}\|_{L^{2}(\Om_T)}^2\\
& \le  A_{\epsilon,\alpha,q}  \mathcal{C}_{E,l}^{2\gamma+2}\|v\|_{L^{q^{\star\star}}(\Om_T)}^{(\theta+1)(2\gamma+2)}+\epsilon\|u_{k,l}\|_{L^{q^{\star\star}}(\Om_T)}^{2\gamma+2}+\frac{\alpha}{2}\||\nabla u_{k,l}||w|^{\gamma}\|_{L^{2}(\Om_T)}^2
\end{split}
\ee
where $\mathcal{C}_{E,l}=(\int_{\Om_T\cap\{|u_{k,l}|>l\}}|E|^r)^\frac1r$, we have used H\"older inequality with exponents
\[
1=\frac1r+\frac{\theta+1}{q^{\star\star}}+\frac12+\frac{\gamma}{q^{\star\star}},
\]
and, at last, Young inequality. For the second integral we have
\be\label{ciop}
\frac{1}{2\gamma+1}\iqT|f||u_{k,l}|^{2\gamma+1}\le A_{q} \mathcal{C}_{f,l}\|u_{k,l}\|_{L^{q^{\star\star}}(\Om_T)}^{2\gamma+1}\le A_{q,\epsilon}  \mathcal{C}_{f,l}^{2\gamma+2}+\epsilon \|u_{k,l}\|_{L^{q^{\star\star}}(\Om_T)}^{2\gamma+2}
\ee
where $\mathcal{C}_{f,l}=(\int_{\Om_T\cap\{|u_{k,l}|\ge l\}}|f|^q)^\frac1q$ and we have used the identity
\[
(2\gamma+1)q'=q^{\star\star}.
\]
Plugging estimates \eqref{cip} and \eqref{ciop} in \eqref{30-10}, we deduce that
\[
\begin{split}
\sup_{\tau\in(0,T)}\io |u_{k,l}(\tau)|^{2\gamma+2}+ \iqT |\nabla u_{k,l}|^{2}|w|^{2\gamma}\\
 \le C_{\alpha,q} \left[\mathcal{C}_{E,l}\|v\|_{L^{q^{\star\star}}(\Om_T)}+ \mathcal{C}_{f,l}+\epsilon\|u_{k,l}\|_{L^{q^{\star\star}}(\Om_T)}+\|u_0\|_{L^{q^{\star\star}\frac{N}{N+2}}(\Omega)}  \right]^{2\gamma+2}.
\end{split}
\]
Using Gagliardo-Nirenberg inequality, it follows that
\[
\|u_{k,l}\|_{L^{q^{\star\star}}(\Om_T)}\le C_{\alpha,q} \left[\mathcal{C}_{E,l}\|v\|_{L^{q^{\star\star}}(\Om_T)}+ \mathcal{C}_{f,l}+\epsilon\|u_{k,l}\|_{L^{q^{\star\star}}(\Om_T)}+\|u_0\|_{L^{q^{\star\star}\frac{N}{N+2}}(\Omega)}  \right],
\]
Choosing $\epsilon$ small enough and taking the limit as $k\to \infty$, we obtain \eqref{stimalk}.
\end{proof}

We shall need the following property.
\begin{lemma}\label{auxiliaryresult}
Take $\delta,K>0$ and $\theta>0$. If 
\[
K\le K_{\delta}=\left(\frac{1}{\delta(\theta+1)}\right)^{\frac{1}{\theta}}\frac{\theta}{\theta+1},
\]
then, for any $s\in [0,R)$ with $R= (\delta(\theta+1))^{-1/ \theta }$, it follows that
\be
\delta s^{\theta+1}+K\le R.
\ee
\end{lemma}
\begin{proof}
 For $s\in(0,R)$ and $K\le K_{\delta}$ we clearly have that $\delta s^{\theta+1}+K\le \delta R^{\theta+1}+K_{\delta}$. Using the definition of $R$ and $K_{\delta}$ it follows that
\[
\begin{split}
\delta R^{\theta+1}+K_{\delta}=&\delta \left(\frac{1}{\delta(\theta+1)}\right)^{1+\frac1{\theta}}+\left(\frac{1}{\delta(\theta+1)}\right)^{\frac{1}{\theta}}\frac{\theta}{\theta+1}\\
=&\left(\frac{1}{\delta(\theta+1)}\right)^{\frac{1}{\theta}}\left(\frac{1}{\theta+1}+\frac{\theta}{\theta+1}\right)=R
\end{split}
\]
\end{proof}

We conclude with the following proof.

\begin{proof}[Proof of Theorem \ref{thmSte}]
We apply Shauder's fixed point Theorem to the map defined in \eqref{mappa}.

\textbf{Existence of an invariant ball} Having in mind estimate \eqref{stimalk}, let us set now
\[
\delta= C_{\alpha,q}\|E\|_{L^{r}(\Om_T)} \ \ \ \mbox{and} \ \ \ K= C_{\alpha,q}\left[\|f\|_{L^{q}(\Om_T)} + \|u_0\|_{L^{q^{\star\star}\frac{N}{N+2}}(\Omega)}\right],
\]
so that, using the same notation of Lemma \ref{auxiliaryresult},
\[
R=(C_{\alpha,q}\|E\|_{L^{r}(\Om_T)}(\theta+1))^{-1/ \theta }.
\]
Under assumption \eqref{smallness}, we can apply Lemma \ref{auxiliaryresult} to conclude that, 
\[
\|v\|_{L^{q^{\star\star}}(\Om_T)}\le R \Rightarrow \|u\|_{L^{q^{\star\star}}(\Om_T)}\le \delta \|v\|_{L^{q^{\star\star}}(\Om_T)}^{\theta+1}+ K\le  R
\]
namely
\[
 \mathscr F (B_{R})\subset B_{R}.
\]

\textbf{Continuity  of the map $ \mathscr F$} Assume $v_n\to v$ in $L^{q^{\star\star}}(\Om_T)$ and set $u_n= \mathscr F(v_n)$ and $u= \mathscr F(v)$. The difference $z_n=u_n-u$ solves
\[
\begin{cases}
\partial_t z_n-\div(M\nabla z_n)= -\div\left[E (|v_n|^{\theta}v_n- |v|^{\theta } v)\right]& \quad  \mbox{in } \Om_T,\\
z_n= 0 & \quad \mbox{on }  (0,T) \times \partial \Omega ,\\
z_n(x,0)=0 & \quad \mbox{in } \Omega.
\end{cases}
\]
Therefore, thanks to Step 1, we have that $z_n$ satisfies the following a priori estimate
\[
\|z_n\|_{L^{q^{\star\star}}(\Om_T)}\le C_{\alpha,q} \left( \iqT   (|v_n|^{\theta} v_n-|v|^{\theta } v)^{\frac{q^{\star\star}}{\theta+1}}  \right)^{\frac{\theta+1}{q^{\star\star}}}.
\]
Since, up to a subsequence, the dominated convergence theorem implies that the right hand side above tends to zero, we have that $z_n$ goes to zero in $L^{q^{\star\star}}(\Om_T)$. Being the argument independent of the considered subsequence, the continuity of the map $v\to  \mathscr F(v)$ follows.\\

\textbf{Compactness} Let us now take a sequence $v_n$ that is bounded in $L^{q^{\star\star}}(\Om_T)$ and set, as before, $u_n= \mathscr F (v_n)$. In order to show that the map $v\to  \mathscr F(v)$ is compact, we shall prove that there exists a not relabeled subsequence of $u_n$ that strongly converges in $L^{q^{\star\star}}(\Om_T)$.\\
Let us recall that $\un\in L^2(0,T,W^{1,2}_0(\Omega))$ satisfies in weak sense.
\[
\partial_t u_n-\div(M(t,x)\nabla u_n)= -\div(E(t,x) |v_n|^{\theta} v_n)+f(t,x),
\]
with $-\div(E(t,x) |v_n|^{\theta} v_n)+f(t,x)$ uniformly bounded in $L^2(0,T,W^{-1,2}(\Omega))$.
This implies that the sequence $u_n$ is bounded in the space 
$
W_2(0,T)$
and, by the Aubin-Lions' compactness result, we deduce that, up to a subsequence, $\un$ strongly converge to some $\zeta$ in $L^2(\Om_T)$ and $\un\to \zeta$ a.e. in $\Om_T$. Using \eqref{stimalk} with $v_n$, we deduce that
\[
\begin{split}
\|G_l(u_n)\|_{L^{q^{\star\star}}(\Om_T)}\le&  C_{\alpha,q} \left[\left(\int_{\Om_T\cap\{|u_n|>l\}}|E|^r\right)^\frac1r\|v_n\|_{L^{q^{\star\star}}(\Om_T)}^{\theta+1}+ \left(\int_{\Om_T\cap\{|u_n|>l\}}|f|^q\right)^\frac1q \right]\\
\le & \tilde C\left[\left(\int_{\Om_T\cap\{|u_n|>l\}}|E|^r\right)^\frac1r+ \left(\int_{\Om_T\cap\{|u_n|>l\}}|f|^q\right)^\frac1q \right]=\omega(l).
\end{split}
\]
Since $\un$ strongly converges in $L^2(\Om_T)$, we have that $\omega(l)$ goes to zero as $l$ diverges. Therefore, for any measurable set $\mathcal E\subset \Om_T$,
\[
\int_{\mathcal E}|u_n|^{q^{\star\star}}\le \int_{\mathcal E}|T_l(u_n)|^{q^{\star\star}}+\io |G_l(u_n)|^{q^{\star\star}}\le l^{q^{\star\star}}|\mathcal E|+\omega^{q^{\star\star}}(l).
\]
This implies that the sequence $|u_n|^{q^{\star\star}}$ is uniformly equi-integrable: indeed, for any $\epsilon>0$, there exists $l$ large enough such that $\omega(l)^{q^{\star\star}}\le \frac{\epsilon}2$; setting $\delta= \frac{\epsilon}{2l^{q^{**}}}$ we have that
\[
\int_{\mathcal E}|u_n|^{q^{**}}\le \epsilon \ \ \ \forall \mathcal E \ : \ |\mathcal E|<\delta, \ \forall n\in\mathbb{N} .
\]
Since $\un\to \zeta$ a.e. in $\Om_T$, Vitali Theorem assures us that $u_n\to \zeta$ in $L^{q^{\star\star}}(\Om_T)$.\\

Therefore, Shauder's fixed point Theorem provides existence of a solution.Uniqueness follows by Theorem \ref{comparison}.
\end{proof}

\section*{Acknowledgments}
S.B. has been supported by the Austrian Science Fund (FWF) project 10.55776/ESP9.   F.F. has been supported  by PRIN Project 2022HKBF5C - PNRR Italia
Domani, funded by EU Program NextGenerationEU.
G. Z. has been supported by  Progetto FRA 2022 \lq\lq Groundwork and OptimizAtion Problems in Transport” from the University of Naples Federico II. 
The authors are members of Gruppo Nazionale per l'Analisi Matematica, la Probabilit\`a e le loro Applicazioni (GNAMPA) of INdAM, and have been  supported by the GNAMPA-INdAM Projects 2025  (CUP E5324001950001).

\end{document}